\newif\ifpreprint
\preprinttrue   

\ifpreprint
\documentclass[12pt]{article}

\usepackage{fullpage}

\usepackage{hyperref}
\hypersetup{colorlinks=true, citecolor=cyan, linkcolor=cyan, urlcolor=cyan}

\usepackage[
backend=biber,
style=ieee,
sortcites=false,
sorting=nyt
]{biblatex}
\newtoggle{bbx:datemissing}

\renewbibmacro*{date}{\toggletrue{bbx:datemissing}%
}

\renewbibmacro*{issue+date}{%
  \toggletrue{bbx:datemissing}%
  \iffieldundef{issue}
    {}
    {\printtext[parens]{%
       \printfield{issue}}}%
  \newunit}

\newbibmacro*{date:print}{%
  \togglefalse{bbx:datemissing}%
  \printdate}

\renewbibmacro*{chapter+pages}{%
  \printfield{chapter}%
  \setunit{\bibpagespunct}%
  \printfield{pages}%
  \newunit
  \usebibmacro{date:print}%
  \newunit}

\renewbibmacro*{note+pages}{%
  \printfield{note}%
  \setunit{\bibpagespunct}%
  \printfield{pages}%
  \newunit
  \usebibmacro{date:print}%
  \newunit}

\renewbibmacro*{addendum+pubstate}{%
  \iftoggle{bbx:datemissing}
    {\usebibmacro{date:print}}
    {}%
  \printfield{addendum}%
  \newunit\newblock
  \printfield{pubstate}}

\else
\documentclass[pdflatex,sn-mathphys-num]{sn-jnl}
\fi

\usepackage{graphicx}
\usepackage{pgfplots}
\usepackage{subcaption}
\usepackage{amsmath,amssymb,amsthm,enumitem,mathtools}

\usepackage{multirow}
\usepackage{algorithm}
\usepackage[noend]{algpseudocode}

\algnewcommand\algorithmicinitialization{\textbf{Initialization:}}
\algnewcommand\Initialization{\item[\algorithmicinitialization]}

\usepackage{csvsimple}	
\usepackage{booktabs}   
\usepackage{adjustbox}  
\usepackage{calc}

\newcommand{\bnote}[1]{}
\newcommand{\anote}[1]{}
\newcommand{\ynote}[1]{}
\newcommand{\snote}[1]{}
\ifpreprint
\renewcommand{\bnote}[1]{\textcolor{cyan}{\textbf{[B: #1]}}}
\renewcommand{\anote}[1]{\textcolor{magenta}{\textbf{[A: #1]}}}
\renewcommand{\ynote}[1]{\textcolor{black}{\textbf{[Y: #1]}}}
\renewcommand{\snote}[1]{\textcolor{brown}{\textbf{[S: #1]}}}
\fi

\usepackage[xindy, acronyms]{glossaries}   
\makeglossaries

\newcommand{\reals}{\mathbb{R}}

\newcommand{\argmin}{\mathop{\rm argmin}}

\newcommand{\NN}{P}
\newcommand{\SOS}{\mathbf{\Sigma}}
\newcommand{\HNN}{\widehat{P}}
\newcommand{\HSOS}{\widehat{\mathbf{\Sigma}}}
\newcommand{\DiSOS}{\mathbf{\Sigma}}
\newcommand{\DiSOSL}{\text{NC}}
\newcommand{\defeq}{\coloneq}

\newcommand{\st}{\text{s.t.}}
\newcommand{\cone}{\operatornamewithlimits{cone}}
\newcommand{\conv}{\operatornamewithlimits{conv}}

\newcommand{\mcQ}{\mathcal{Q}}

\newcommand{\mcC}{\mathcal{C}}

\newcommand{\mcN}{\mathcal{N}}
\newcommand{\mbS}{\mathbb{S}}
\newcommand{\mcT}{\mathcal{T}}
\newcommand{\mcD}{\mathcal{D}}
\newcommand{\mcV}{\mathcal{V}}
\newcommand{\mcPN}{\mathcal{PN}}

\newtheorem{theorem}{Theorem}
\newtheorem{lemma}[theorem]{Lemma}  
\newtheorem{definition}{Definition}
\theoremstyle{remark}
\newtheorem{remark}{Remark}

\newacronym{LO}{LO}{linear optimization}
\newacronym{QO}{QO}{quadratic optimization}
\newacronym{MIQO}{MIQO}{mixed-integer quadratic optimization}
\newacronym{MIO}{MIO}{mixed-integer optimization}
\newacronym{MILO}{MILO}{mixed-integer linear optimization}
\newacronym{MINLO}{MINLO}{mixed-integer nonlinear optimization}
\newacronym{sBB}{sBB}{spacial branch and bound}
\newacronym{NLO}{NLO}{nonlinear optimization}
\newacronym{PWA}{PWA}{piecewise affine}
\newacronym{SVM}{SVM}{support vector machines}
\newacronym{ReLU}{ReLU}{rectified linear unit}
\newacronym{CPU}{CPU}{central processing unit}
\newacronym{GPU}{GPU}{graphics processing unit}
\newacronym{MPC}{MPC}{model predictive control}
\newacronym{ADMM}{ADMM}{alternating direction method of multipliers}
\newacronym{ADP}{ADP}{approximate dynamic programming}
\newacronym{FPGA}{FPGA}{field-programmable gate array}



\newcommand{\abstracttext}{We introduce the concept of \emph{disjunctive sum of squares} for certifying nonnegativity of polynomials. Unlike the popular sum of squares approach where nonnegativity is certified by a single algebraic identity, the disjunctive sum of squares approach certifies nonnegativity with multiple algebraic identities which can be found in parallel. Our main result is a disjunctive Positivstellensatz proving that we can keep the degree of each algebraic identity as low as the degree of the polynomial whose nonnegativity is in question. Based on this result, we construct a semidefinite programming based converging hierarchy of lower bounds for the problem of minimizing a polynomial over a compact basic semialgebraic set, where the size of the largest semidefinite constraint is fixed throughout the hierarchy. We further prove a second disjunctive Positivstellensatz which leads to an optimization-free hierarchy for polynomial optimization. We specialize this result to the problem of proving copositivity of matrices. Finally, we describe how the disjunctive sum of squares approach can be combined with a branch-and-bound algorithm and we present numerical experiments on polynomial, copositive, and combinatorial optimization problems.}

\begin{document}

\ifpreprint
\title{ Disjunctive Sum of Squares }
\author{Amir Ali Ahmadi\thanks{Princeton University, Operations Research and Financial Engineering.
AAA was partially supported by the Sloan Fellowship, the Princeton AI Lab Seed Grant, the Princeton SEAS Innovation Grant, and a Research Gift in Mathematical Optimization.
YH was partially supported by the Princeton AI Lab Seed Grant, the Princeton SEAS Innovation Grant, and the IBM PhD Fellowship.},
 Sanjeeb Dash\thanks{IBM TJ Watson Research Center.},
Yixuan Hua\footnotemark[1],
Bartolomeo Stellato\footnotemark[1]}
\date{}
\maketitle

\begin{abstract}
\abstracttext
\end{abstract}
\fi


\section{Introduction}
A polynomial $p:\mathbb{R}^n\rightarrow\mathbb{R}$ is said to be \emph{nonnegative} if $p(x)\ge 0$ for all $x\in\reals^n$. 
The problem of optimizing over the intersection of the cone of nonnegative polynomials of a given degree with an affine subspace 
is fundamental to optimization and algebraic geometry, with broad applications in polynomial optimization~\cite{lasserre2001global}, control systems and robotics~\cite{henrion2005positive,tedrake2010lqr}, probability theory~\cite{bertsimas2005optimal}, and statistics~\cite{curmei2023shape}, to name a few; see the textbooks~\cite{lasserre2009moments,blekherman2012semidefinite,lasserre2015introduction}, or survey papers~\cite{laurent2009sums,hall2020applications} for many other applications and further context.  

A polynomial $p:\mathbb{R}^n\rightarrow\mathbb{R}$ is said to be a \emph{sum of squares} (sos) if it can be expressed as
$p(x)=\sum_i q_i^2(x)$ for some polynomials $q_i:\mathbb{R}^n\rightarrow\mathbb{R}$. 
A sum of squares decomposition of a polynomial provides a single algebraic identity that certifies its nonnegativity.
Moreover, as is well known by now, such a decomposition (assuming it exists) can be obtained by solving a semidefinite program (SDP)~\cite{Parrilo2000,lasserre2001global,parrilo2003semidefinite}.


Throughout the paper, we denote the set of nonnegative polynomials and sos polynomials in $n$ variables and degree at most $d$ by $\NN_{n,d}$ and $\SOS_{n,d}$, respectively. 
It has been known since the seminal work of Hilbert~\cite{Hilbert1888} that not every nonnegative polynomial is a sum of squares. In fact, Hilbert established that $\NN_{n,d}$ and $\SOS_{n,d}$ coincide only in three specific cases: when $n=1$, $d=2$, or $(n,d)=(2,4)$. The first example of a nonnegative polynomial which is not sos was constructed by Motzkin many years later~\cite{Motzkin1967}:
\begin{equation}\label{eq:motzkin_nonhomo}
M(x_1,x_2)=x_1^4x_2^2+x_1^2x_2^4-3x_1^2x_2^2+1.
\end{equation}


Despite the existence of such examples, it turns out that a proof of nonnegativity based on a single algebraic identity is always possible though at the expense of increasing the degree of the sos polynomials that appear in the identity. More specifically, it follows from Artin's solution to Hilbert's 17th problem~\cite{Artin1927} that for every nonnegative polynomial $p$, there exists a nonzero sos polynomial $q$, such that the product $pq$ is sos (and hence nonnegativity of $p$ is algebraically certified by a representation as a ratio of two sos polynomials). For example, for the Motzkin polynomial, we have the following degree-8 identity:
\begin{equation}\label{eq:Motzkin.with.multiplier}
(x_1^2+x_2^2)M(x_1,x_2)=x_1^2(1-x_2^2)^2+x_2^2(1-x_1^2)^2+x_1^2x_2^2(x_1^2+x_2^2-2)^2.    
\end{equation}
By utilizing Artin's theorem, one can design a complete SDP-based hierarchy for proving nonnegativity of a polynomial $p$. In level $j$ of this hierarchy, one searches for a nonzero sos polynomial~$q$ of degree~$2j$ that makes the product $pq$ also sos.\footnote{Unlike the approach we later propose in this paper, one cannot use this SDP hierarchy to \emph{impose} a nonnegativity constraint on a polynomial $p$. This is because of the bilinear matrix inequality that arises from the sos constraint on the product of the two unknown polynomials $p$ and $q$.}
However, the degree of the polynomial $pq$ can be much higher than the degree of~$p$ (see, e.g.,~\cite{sos_on_hypercube}) and is upper bounded currently only by a tower of five exponentials in the number of variables and degree of $p$~\cite{lombardi2020elementary}. Since the size of the SDP underlying an sos 
constraint grows rapidly with the degree of the 
sos polynomial, 
the SDP hierarchy based on Artin's theorem can quickly become impractical in certain applications. This challenge also arises in the context of other sum of squares based SDP hierarchies (e.g., in~\cite{Parrilo2000,lasserre2001global,parrilo2003semidefinite}) which rely on related seminal theorems of algebraic geometry such as the
Positivstellens\"atze of Stengle~\cite{stengle1974nullstellensatz}, Schm\"udgen~\cite{Schmudgen1991}, or Putinar~\cite{putinar1993positive}.

With this context in mind, the central question that motivates our research is the following: \emph{``Can one lower the complexity of a sum of squares based proof of polynomial nonnegativity by constructing several algebraic identities instead of a single one?''} Our primary notion of complexity here is the degree of the sos polynomials that appear in the proof of nonnegativity. However, our results also simplify the structure of these sos polynomials making the search for them amenable to cheaper classes of conic programs, such as linear programs. As an example, consider the following proof of nonnegativity of the Motzkin polynomial using two degree-6 algebraic identities:
\begin{equation}\label{eq:Motzkin.disos.proof}
\begin{aligned}
M(x_1,x_2) &= (1-x_1x_2)^2 + (x_1^2x_2-x_1x_2^2)^2 + 2x_1x_2(1-x_1x_2)^2\\
    & = (1+x_1x_2)^2 + (x_1^2x_2+x_1x_2^2)^2 - 2x_1x_2(1+x_1x_2)^2. 
    \end{aligned}    
\end{equation}
Indeed, for every point $(x_1,x_2)\in\reals^2$, either $x_1x_2\ge 0$ or $x_1x_2<0$. When $x_1x_2\ge 0$, nonnegativity of $M(x_1,x_2)$ follows from the first identity. Similarly, when $x_1x_2<0$ or equivalently $-x_1x_2>0$, nonnegativity of $M(x_1,x_2)$ follows from the second identity. We call a proof of this type a \emph{disjunctive sum of squares} proof of nonnegativity. Comparing the two algebraic proofs in (\ref{eq:Motzkin.with.multiplier}) and (\ref{eq:Motzkin.disos.proof}), we see that the former involves an sos polynomial of degree 8 while the latter involves two sos polynomials of degree 6, which is the same as the degree of the Motzkin polynomial. From a computational standpoint, this means that the size of the semidefinite constraint underlying the second proof is smaller.\footnote{We recall that, imposing an sos constraint on a polynomial in $n$ variables and degree $2d$ results in a positive semidefinite constraint on a matrix of size ${{n+d}\choose{d}}\times {{n+d}\choose{d}}$ (see, e.g.,~\cite{Parrilo2000, parrilo2003semidefinite}).}
Moreover, the sos polynomials appearing in (\ref{eq:Motzkin.disos.proof}) have a special structure (known as ``diagonal dominance''; see~\cite{aaa2018dsos}), that enables the search for this particular disjunctive sum of squares proof to be carried out by solving a linear program. This is, in general, much cheaper than solving a semidefinite program.


Our main contribution in this paper is to formalize the concept of disjunctive sum of squares proofs of nonnegativity and show that \emph{low-degree} proofs of this type always exist under standard assumptions. We outline our contributions in some more detail next.



\subsection{Organization and Main Contributions}
The remainder of this paper is organized as follows. In Section~\ref{sec:disos}, we define the notion of a disjunctive sum of squares proof of nonnegativity of a polynomial. 
We also show that a natural alternating maximization approach based on SDP can find {low-degree disjunctive sum of squares proofs of nonnegativity of several classical non-sos polynomials}.

In Section~\ref{sec:disos_psatz}, we prove a ``disjunctive Positivstellensatz'', which shows that the degree of the sos polynomials appearing in a disjunctive sum of squares proof can be bounded above by the degree of the polynomial whose nonnegativity is being certified. Sections~\ref{sec:local_psatz} and \ref{sec:global_psatz} are devoted to the proof of this statement. Based on this result, Section~\ref{sec:disos_grid} introduces an SDP-based hierarchy, with semidefinite constraints of \emph{fixed size}, which yields 
a converging sequence of 
lower bounds on the minimum value of a homogeneous polynomial over the unit sphere. Our analysis also yields a convergence rate for this hierarchy.

In Section~\ref{sec:disos_new}, we prove a second disjunctive Positivstellensatz, which shows that a proof of nonnegativity of a polynomial can be given by simply checking nonnegativity of the coefficients of the same polynomial under several linear changes of coordinates. These linear transformations are explicit and derived from disjunctions based on simplicial partitions of the Euclidean space. An upper bound on the number of linear transformations is provided.



In Section~\ref{sec:constrained_opt}, we demonstrate how our disjunctive sum of squares approach can be applied to constrained polynomial optimization problems. 
In Section~\ref{sec:hierarchy}, 
using a reduction from~\cite{aaa2019}, we construct a semidefinite programming based converging hierarchy of lower bounds for the general problem of minimizing a polynomial over a compact basic semialgebraic set, where the size of the largest semidefinite constraint is fixed throughout the hierarchy.
In Section~\ref{sec:copositive}, 
we focus on the case where the constraint set is the unit simplex, and we provide a specialized disjunctive Positivstellensatz for certifying copositivity of matrices.



In Section~\ref{sec:bnb}, we show how one can more practically search for disjunctive sum of squares proofs by dynamically partitioning the Euclidean space through a 
spatial branch-and-bound framework. As examples, 
%
we propose branch-and-bound algorithms for minimizing a homogeneous polynomial over the unit sphere and for certifying copositivity of matrices. In Section~\ref{sec:num}, we present some numerical experiments. 
Finally, Section~\ref{sec:discussion} concludes our paper and highlights potential future directions.

\section{SOS with Respect to an Algebraic Disjunction}\label{sec:disos}
In this section, we introduce the concept of sum of squares with respect to an algebraic disjunction, which is central to our framework. 
Throughout the paper, we denote the degree of a polynomial $p$ by $\deg p$.

\begin{definition}[Algebraic Disjunction]\label{def:disjunction}
Let $r,n_1,\ldots,n_r$ be positive integers. We say that a collection of polynomials $\mcD=\{\{q_{k,j}\}_{j=1,\ldots, n_k},\quad k=1,\ldots,r\}$ forms an \emph{algebraic disjunction} of $\reals^n$ if 
\[\mathbb{R}^n=\bigcup\limits_{k=1}^r\Omega_{k},\ \text{where }\Omega_k=\{x\in\mathbb{R}^n\mid q_{k,j}(x)\ge 0,\quad j=1,\ldots,n_k\}.\]
\end{definition}
\begin{definition}\label{def:disos}
Let $p$ be a polynomial in $n$ variables and degree at most $d$, $\bar{d}\ge d$ be an integer, and $\mcD=\{\{q_{k,j}\}_{j=1,\ldots, n_k},\quad k=1,\ldots,r\}$ be an algebraic disjunction of $\reals^n$.
We say that $p$ is \emph{sos of order $\bar{d}$ with respect to }$\mcD$ \emph{(or $p$ is sos$^{\bar{d}}_\mathcal{D}$)} if there exist sos polynomials $s_{k,j}$ for $k=1,\ldots, r$ and $j=0,1,\ldots,n_k$, such that 
\begin{equation}\label{def:sos_ad}
    p(x)=s_{k,0}(x)+\sum_{j=1}^{n_k} s_{k,j}(x)q_{k,j}(x),\quad  k=1,\ldots, r,
\end{equation}
where $\deg s_{k,0}\le \bar{d}$ and $\deg(s_{k,j}q_{k,j})\le \bar{d}$ for $k=1,\ldots,r$ and  $j=1,\ldots,n_k$. 
We denote the set of sos$^{\bar{d}}_\mcD$ polynomials in $n$ variables and degree at most $d$ as $\DiSOS_{n,d}^{\bar{d}}(\mcD)$.
\end{definition}

Observe that for any algebraic disjunction $\mcD$ and integer $\bar{d}\ge d$, the set $\DiSOS^{\bar{d}}_{n,d}(\mcD)$ is a convex cone and lies between $\SOS_{n,d}$ and $\NN_{n,d}$. 
Indeed, by setting $s_{k,j}=0$ for $k=1,\ldots,r$ and $j=1,\ldots,n_k$, we note that $\SOS_{n,d}\subseteq\DiSOS^{\bar{d}}_{n,d}(\mcD)$. This inclusion can be strict even when $\bar{d}=d$; e.g., the Motzkin polynomial in~\eqref{eq:motzkin_nonhomo} belongs to $\DiSOS_{2,6}^{6}(\{ \{+x_1x_2\},\{-x_1x_2\}\})$, as shown by the identities in~\eqref{eq:Motzkin.disos.proof}. 
To see that $\DiSOS^{\bar{d}}_{n,d}(\mcD)\subseteq \NN_{n,d}$, note that for any point $x\in\Omega_k$, we have $q_{k,j}(x)\ge 0$ for all $j=1,\ldots,n_k$, and therefore, in view of~\eqref{def:sos_ad},~$p(x)\ge 0$. 
Consequently, $p$ is nonnegative over each subregion $\Omega_k$, and thus nonnegative over~$\reals^n$. We call the $r$ algebraic identities in~\eqref{def:sos_ad} a disjunctive sos proof of nonnegativity of~$p$. 

From a computational standpoint, for any fixed $\bar{d}$, linear optimization over the intersection of an affine subspace with the cone $\DiSOS_{n,d}^{\bar{d}}(\mcD)$ can be carried out by an SDP of size polynomial in $n$. We also note that for implementation purposes, there is no need to include $s_{k,0}$ as a decision variable; instead, one can impose the constraint that~${p(x)-\sum_{j=1}^{n_k} s_{k,j}(x)q_{k,j}(x)\in\SOS_{n,\bar{d}}}$ for $k=1,\ldots,r$. 


Unlike the well-known hierarchies for proving nonnegativity of a polynomial (e.g., based on the Positivstellens\"atze of Artin~\cite{Artin1927}, Reznick~\cite{reznick1995uniform}, Stengle~\cite{stengle1974nullstellensatz, Parrilo2000}, or Putinar~\cite{putinar1993positive, lasserre2001global}) that increase the degrees of the sos polynomials involved in the proof at each level of the hierarchy, our approach keeps the degrees of these sos polynomials fixed (and low). 
We instead increase the number of subsets in the algebraic disjunction $\mcD$ (see Theorem~\ref{cor:strictly_pos} and Theorem~\ref{cor:strictly_pos_new}). 

A natural question is how to search for disjunctive proofs of nonnegativity of a polynomial~$p$; i.e., for the identities in~\eqref{def:sos_ad}. Note that as written, the joint search for the polynomials $q_{k,j}$ that define the algebraic disjunction $\mathcal{D}$ and the sos multipliers $s_{k,j}$ cannot be formulated as a convex optimization problem. An obvious workaround to this issue is to alternate the search for these decision variables. The details of this approach are described in the rest of this section together with some examples. In Section~\ref{sec:disos_psatz} and Section~\ref{sec:disos_new}, we present a more theoretically sound approach to finding disjunctive proofs of nonnegativity. We show that by appropriately fixing a family of algebraic disjunctions, we can always construct low-degree proofs of nonnegativity using convex optimization.





\paragraph{An alternating maximization algorithm.}



In the joint search for the polynomials~$q_{k,j}$ and the sos multipliers~$s_{k,j}$ in~\eqref{def:sos_ad}, a challenge is to ensure that the polynomials~$q_{k,j}$ form a valid algebraic disjunction (see Definition~\ref{def:disjunction}). 
A simple way to guarantee this is to take $\ell$ polynomials $h_1,\ldots, h_{\ell}$, let $r=2^{\ell}$, $\varepsilon \in \{-1,1\}^{r \times \ell}$ with all rows of $\varepsilon$ distinct, and let 
\[q_{k,j}=\varepsilon_{k,j}h_j,\quad k=1,\ldots,r, \quad j=1,\ldots,\ell.\]
With this construction, it is easy to see that the union of the corresponding $r$ subregions $\Omega_1,\ldots,\Omega_r$ as defined in Definition~\ref{def:disjunction} is $\reals^n$.
Building on this, given a polynomial $p$ of degree $d$ and an integer $\bar{d}\ge d$, the problem boils down to finding polynomials~$h_j$ and sos multipliers~$s_{k,j}$ such that
\begin{equation}\label{eq:am_joint_feas}
		p(x) =s_{k,0}(x)+\sum\limits_{j=1}^{\ell} \varepsilon_{k,j}h_j(x)s_{k,j}(x),\quad  k=1,\ldots,r,    
\end{equation}
where $\deg s_{k,0}\le \bar{d}$, 
$\deg(h_js_{k,j})\le \bar{d}$ for $k=1,\ldots,r$ and $j=1,\ldots,\ell$. 
To approach this problem using convex optimization, we alternate between the search for the polynomials~$h_j$ and the sos multipliers~$s_{k,j}$. 
To track progress across iterations, we always maximize a lower bound on~$p$ that can be certified with identities of the form~\eqref{eq:am_joint_feas}. 
More concretely, we alternate between the following two steps, which can both be reformulated as a semidefinite program:
\begin{enumerate}
    \item 
    Given polynomials~$h_j^*$ for $j=1,\ldots,\ell$, solve the SDP
    \begin{equation*}
	\begin{array}{cl}
\max\limits_{\gamma,s_{k,j}} & \gamma\\
		\st & p(x)-\gamma =s_{k,0}(x)+\sum\limits_{j=1}^{\ell} \varepsilon_{k,j}h_j^*(x)s_{k,j}(x),\quad  k=1,\ldots,r\\
         & \deg s_{k,0}\le \bar{d},\ 
\deg(h_j^*s_{k,j})\le \bar{d},\quad  k=1,\ldots,r,\ j = 1,\ldots, \ell,\\
& s_{k,j} \text{ is sos},\quad  k=1,\ldots,r,\ j = 0,\ldots, \ell.
	\end{array}
\end{equation*}
    \item Given sos polynomials~$s_{k,j}^*$ for $k=1,\ldots,r$ and $j=1,\ldots,\ell$, solve the SDP\footnote{{In fact, both Step~1 and Step~2 can be decomposed into $r$ independent SDPs and solved in parallel; see Remark~\ref{remark:uniform_grid} in Section~\ref{sec:disos_grid}.}} 
     \begin{equation*}
	\begin{array}{cl}
\max\limits_{\gamma,h_j,s_{k,0}} & \gamma\\
		\st & p(x)-\gamma =s_{k,0}(x)+\sum\limits_{j=1}^{\ell} \varepsilon_{k,j}h_j(x)s_{k,j}^*(x),\quad  k=1,\ldots,r\\
         & \deg s_{k,0}\le \bar{d},\ \deg(h_js_{k,j}^*)\le \bar{d},\quad  k=1,\ldots,r,\ j = 1,\ldots, \ell,\\
& s_{k,0} \text{ is sos},\quad  k=1,\ldots,r.
	\end{array}
\end{equation*}   
\end{enumerate}
To run Step 1, we let the polynomials $h_j^*$ for ${j}=1,\ldots, \ell$ be arbitrary (e.g., chosen at random) if we are in the first iteration, or an optimal solution to Step 2 if we are in subsequent odd iterations. Similarly, to run Step 2, we let $s_{k,j}^*$ for $k=1,\ldots,r$ and $j=1,\ldots,\ell$ be an optimal solution to Step 1.



Applying the alternating maximization algorithm to the Motzkin polynomial in~\eqref{eq:motzkin_nonhomo} yields two additional disjunctive sos proofs of nonnegativity. In this example, we set $\ell=1$ and $\bar{d}=6$. To initialize the algorithm, we generate $h_1$ randomly, once of degree 1 and once of degree 4. The alternating maximization algorithm converges in fewer than 20 iterations and produces (after some minor post-processing to polish the coefficients to rational numbers) the following two algebraic identities, respectively: 
\begin{equation}\label{eq:Motzkin.disos.proof2}
 \begin{aligned}
    M(x_1,x_2) &= (1-x_1x_2^2)^2 + (x_1^2x_2-x_1x_2)^2 + 2x_1(x_1x_2-x_2)^2\\
    & = (1+x_1x_2^2)^2 + (x_1^2x_2+x_1x_2)^2 - 2x_1(x_1x_2+x_2)^2,\\
\end{aligned} 
\end{equation}
\begin{equation}\label{eq:Motzkin.disos.proof3}
 \begin{aligned}
M(x_1,x_2) &= \frac{1}{2}(x_1^2x_2+x_2^3-2x_2)^2 + (x_2^2-1)^2  +\frac{1}{2}(x_1^4-x_2^4-2x_1^2+2x_2^2)x_2^2\\
  &=\frac{1}{2}(x_1^3+x_1x_2^2-2x_1)^2 + (x_1^2-1)^2 - \frac{1}{2}(x_1^4-x_2^4-2x_1^2+2x_2^2)x_1^2.
\end{aligned} 
\end{equation}
Figure~\ref{fig:Motzkin} demonstrates the subregions associated with the three degree-6 disjunctive sos proofs of nonnegativity of the Motzkin polynomial in~\eqref{eq:Motzkin.disos.proof},~\eqref{eq:Motzkin.disos.proof2}, and~\eqref{eq:Motzkin.disos.proof3}.


Similarly, we are able to find disjunctive proofs of nonnegativity of several polynomials that are known to \emph{not} be a sum of squares. We present three such certificates below.
\begin{itemize}[noitemsep]
\item Choi-Lam-1~\cite{ChoiLam1976}. Let $ CL_1(x_{1},x_{2},x_{3},x_{4}) = -4x_{1} x_{2} x_{3} x_{4} + x_{1}^{2} x_{2}^{2} + x_{1}^{2} x_{3}^{2} + x_{2}^{2} x_{3}^{2} + x_{4}^{4}$.
Then $CL_1 = s_1 + h\,s_2 = s_3 - h\,s_4$, with $h = x_{1} x_{2}$,
\[
  \begin{array}{rcl@{\qquad}rcl}
    s_1 & = & (x_{1} x_{3} - x_{2} x_{3})^2 +\, (x_{1} x_{2} - x_{4}^{2})^2, & s_2 & = & 2\,(-x_{3} + x_{4})^2, \\
    s_3 & = & (x_{1} x_{3} + x_{2} x_{3})^2 +\, (x_{1} x_{2} + x_{4}^{2})^2, & s_4 & = & 2\,(x_{3} + x_{4})^2.
  \end{array}
\]
\item Choi-Lam-2~\cite{ChoiLam1977}.
Let $CL_2(x_{1},x_{2},x_{3}) = -3x_{1}^{2} x_{2}^{2} x_{3}^{2} + x_{1}^{2} x_{3}^{4} + x_{1}^{4} x_{2}^{2} + x_{2}^{4} x_{3}^{2}$. 
Then $CL_2~=~ s_1 + h\,s_2 = s_3 - h\,s_4$, with $h = x_{1} x_{2}$,
\[
  \begin{array}{rcl@{\qquad}rcl}
    s_1 & = & (-x_{1} x_{3}^{2} + x_{1}^{2} x_{2})^2 +\, (-x_{1} x_{2} x_{3} + x_{2}^{2} x_{3})^2, & s_2 & = & 2\,(-x_{1} x_{3} + x_{2} x_{3})^2, \\
    s_3 & = & (x_{1} x_{3}^{2} + x_{1}^{2} x_{2})^2 +\, (x_{1} x_{2} x_{3} + x_{2}^{2} x_{3})^2, & s_4 & = & 2\,(x_{1} x_{3} + x_{2} x_{3})^2.
  \end{array}
\]
\item Stengle-1~\cite{Stengle1979}. Let $ S_1(x_{1},x_{2},x_{3}) = -2x_{1} x_{2}^{2} x_{3}^{3} + x_{1}^{2} x_{3}^{4} - 2x_{1}^{3} x_{2}^{2} x_{3} + x_{1}^{3} x_{3}^{3} + 2x_{1}^{4} x_{3}^{2} + x_{1}^{6} + x_{2}^{4} x_{3}^{2}$.
Then $S_1 = s_1 + h\,s_2 = s_3 - h\,s_4$, with $h = x_{1} x_{3}$,
\[
  \begin{array}{rcl@{\qquad}rcl}
    s_1 & = & (-x_{1} x_{3}^{2} - x_{1}^{3} + x_{2}^{2} x_{3})^2, & s_2 & = & (x_{1} x_{3})^2, \\
    s_3 & = & \multicolumn{4}{l}{(x_{1} x_{3}^{2} + \tfrac{1}{2}\,x_{1}^{2} x_{3} + \tfrac{7}{10}\,x_{1}^{3} - \tfrac{6}{7}\,x_{2}^{2} x_{3})^2 +\, \tfrac{1}{14}\,(x_{1} x_{2} x_{3})^2} \\
     &  & \multicolumn{4}{l}{+\, \tfrac{7}{20}\,(x_{1}^{2} x_{3} - x_{1}^{3} + \tfrac{5}{7}\,x_{2}^{2} x_{3})^2 +\, \tfrac{4}{25}\,(x_{1}^{3} - \tfrac{5}{7}\,x_{2}^{2} x_{3})^2 +\, \tfrac{1}{196}\,(x_{2}^{2} x_{3})^2,} \\
    s_4 & = & \tfrac{2}{7}\,(-\tfrac{1}{2}\,x_{1} x_{2} + x_{2} x_{3})^2. & & &
  \end{array}
\]
\end{itemize}

\begin{figure}[ht]
\centering
\includegraphics[width=0.33\textwidth]{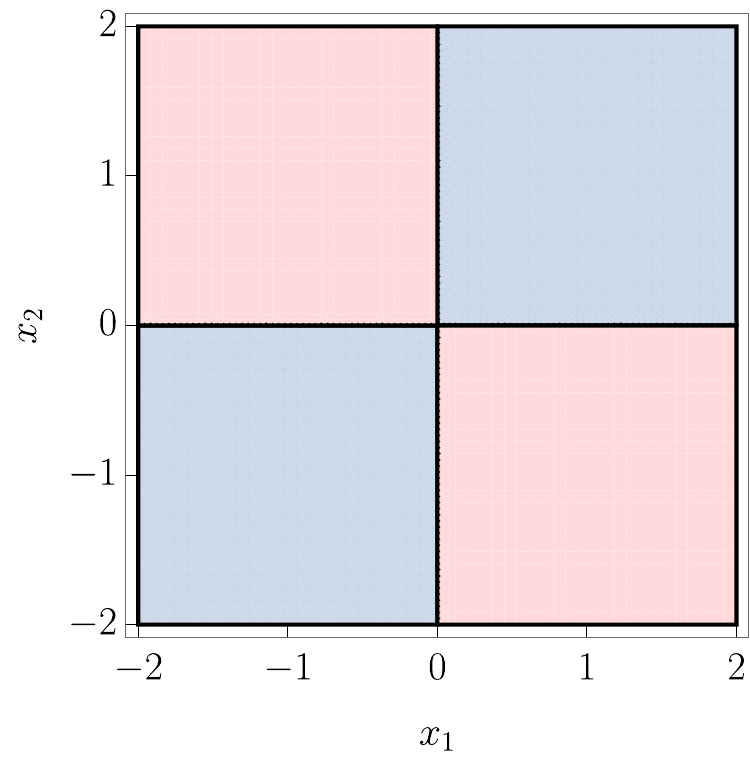}\hfill
\includegraphics[width=0.33\textwidth]{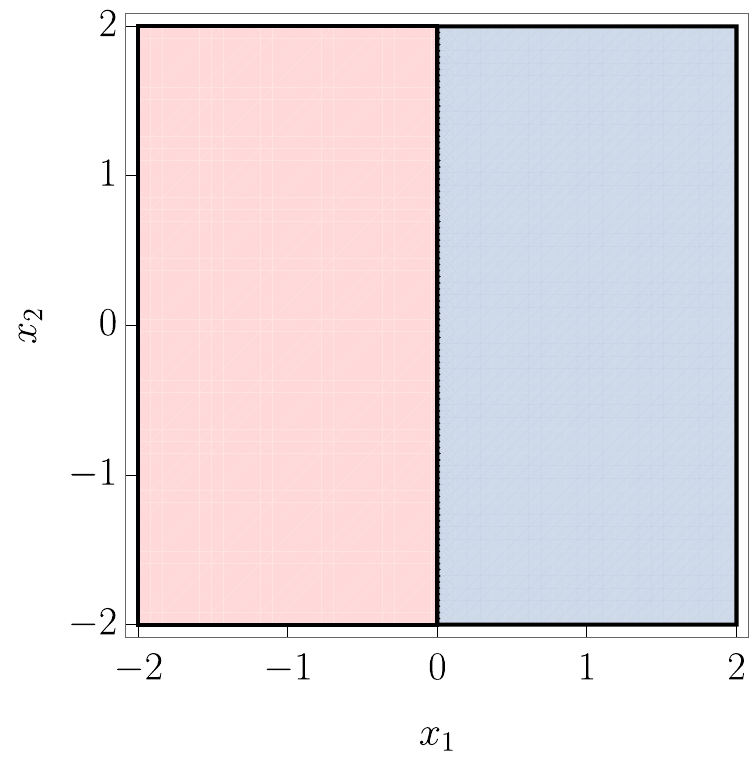}
\includegraphics[width=0.33\textwidth]{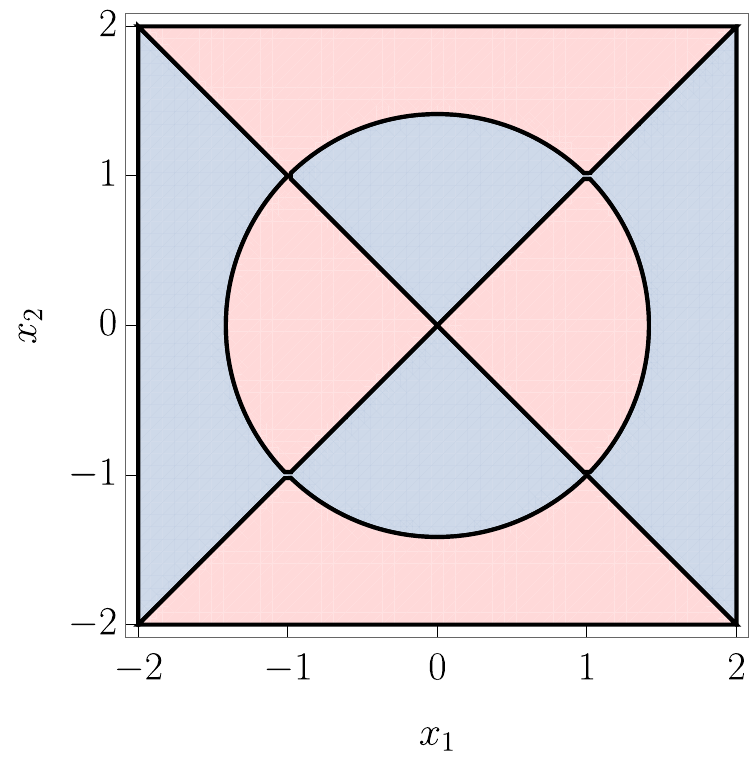}
    \caption{Three different partitions of $\reals^2$ corresponding to three different degree-6 disjunctive sos proofs of nonnegativity of the Motzkin polynomial (see~\eqref{eq:motzkin_nonhomo}). Left: subregions correspond to the identity in~\eqref{eq:Motzkin.disos.proof} and are defined by $x_1x_2\ge 0$ and $-x_1x_2 \ge 0$. 
    Middle: subregions correspond to the identity in~\eqref{eq:Motzkin.disos.proof2} and are defined by $x_1\ge 0$ and $-x_1\ge 0$. 
    Right: subregions correspond to the identity in~\eqref{eq:Motzkin.disos.proof3} and are defined by $x_1^4-x_2^4-2x_1^2+2x_2^2 \ge 0$ and $-(x_1^4-x_2^4-2x_1^2+2x_2^2) \ge 0$.
    }
\label{fig:Motzkin}
\end{figure}

\section{A Disjunctive Positivstellensatz}
\label{sec:disos_psatz}

In this section, we introduce our main result in Theorem~\ref{cor:strictly_pos} which shows that one can a priori fix a family of low-degree algebraic disjunctions with respect to which there always exists a low-degree disjunctive sos proof of nonnegativity.

Let us establish some basic notation related to \emph{forms}, i.e., homogeneous polynomials, which will be useful in the remainder of this paper.
We denote the set of forms, nonnegative forms, and sum of squares forms in~$n$ variables and degree~$d$ by $H_{n,d},\HNN_{n,d},$ and $\HSOS_{n,d}$ respectively. 
Observe that~$\HNN_{n,d}=\NN_{n,d}\cap H_{n,d}$ and $\HSOS_{n,d}=\SOS_{n,d}\cap H_{n,d}$. 
Additionally, we say that a form~$p$ in~$n$ variables is \emph{positive definite} if $p(x)>0$ for all $x\in\reals^n$ and $x\neq0$.
We are now ready to present our main result.
\begin{theorem}[Disjunctive Positivstellensatz]\label{cor:strictly_pos}
For any fixed dimension~$n$ and even degree~$d$, there exists an explicit family of algebraic disjunctions
$\mcD^m$, indexed by $m$, such that for any positive
 definite form $p\in H_{n,d}$, we have $p\in \DiSOS_{n,d}^{d}(\mcD^m)$ for some~$m$. Furthermore, each subset in $\mcD^m$ contains exactly one form of degree at most~$d$.\footnote{Our proof in fact provides several constructions: the forms in the algebraic disjunction may be chosen to have any even degree between~2 and~$d$.}
\end{theorem}

We prove Theorem~\ref{cor:strictly_pos} as follows.
In Section~\ref{sec:local_psatz}, we introduce local Positivstellens\"atze, which provide a single algebraic identity proving nonnegativity of a polynomial in the neighborhood of a given point where it is known to be positive.
In Section~\ref{sec:global_psatz}, we combine these neighborhoods to cover the entire Euclidean space and use the corresponding algebraic identities to construct a disjunctive sos proof of (global) nonnegativity.
In Section~\ref{sec:disos_grid}, we show how this theorem leads to a hierarchy of SDPs---with semidefinite constraints of fixed size---that produces arbitrarily tight lower bounds on the minimum value of a homogeneous polynomial over the unit sphere.

\subsection{Local Positivstellens\"atze}\label{sec:local_psatz}
For a polynomial~$p$ in~$n$ variables and degree~$d$, we denote its homogenization by
$$p_h(x,x_{n+1})\defeq x_{n+1}^dp(x/x_{n+1}).$$
It is straightforward to check that homogenization preserves both nonnegativity and the sos property; that is, $p\in\NN_{n,d}$ (resp., $p\in\SOS_{n,d}$) if and only if $p_h\in\HNN_{n+1,d}$ (resp., $p_h\in\HSOS_{n+1,d}$). Hence, given a general (not necessarily homogeneous) polynomial~$p$, we can construct a disjunctive sos proof of nonnegativity of~$p$ by first obtaining one for $p_h$ and then setting the last variable to 1 in the identities. This justifies our focus on homogeneous polynomials. Moreover, since odd-degree polynomials cannot be nonnegative on the entire Euclidean space, it suffices to consider even-degree forms. 

We now turn to establishing local certificates of nonnegativity. Suppose a polynomial~$p$ takes {a} positive value at some point $x^*\in\mathbb{R}^n$. By continuity, $p$ remains nonnegative within a sufficiently small neighborhood of $x^*$. This observation motivates the development of a local Positivstellensatz: an algebraic identity demonstrating that if the angle between a point $x$ and $x^*$ is sufficiently small, then $p(x)\ge 0$. 

We introduce some notation for the theorems that follow. We denote the unit sphere in~$\reals^n$ by~$\mbS^{n-1}$, and the $n\times n$ identity matrix by $I_n$. 
For a polynomial~$p$, we denote its largest coefficient in absolute value (resp.\ the sum of the absolute values of its coefficients) in the standard monomial basis by $\|p\|_{\infty}$ (resp.\ by $\|p\|_1$).



\begin{theorem}\label{thm:degd_homo}
For a polynomial $p\in H_{n,d}$ with $d$ even, if $p(x^*)>0$ for some $x^*\in \mbS^{n-1}$, then for $t>0$ sufficiently large, $p(x)$ can be written as
\begin{equation}\label{eq:degd_homo}
    p(x)=\left(\frac{p(x^*)}{2}(x^Tx^*)^d-th_d(x;x^*)\right)+\sigma(x),
\end{equation}
where $\sigma(x)\in\HSOS_{n,d}$ and 
\begin{equation}\label{eq:degd_homo_h}
h_d(x;x^*)\defeq\sum_{k=1}^{d/2}\left(\|x\|_2^2-(x^Tx^*)^2\right)^k(x^Tx^*)^{d-2k}.
\end{equation}
%
In particular, the above statement holds for any
\begin{equation}\label{eq:degd_homo_t}
t\ge T(p;x^*)\defeq\sup_{U\in\reals^{n\times n}:U^TU=I_n}\left( \frac{n}{2p(x^*)}\|p(Ux)\|_{\infty}^2 + \|p(Ux)\|_1\right).  
\end{equation}
\end{theorem}

\begin{remark}
The identity in~\eqref{eq:degd_homo} shows the implication that
\begin{equation}\label{eq:imp_homo}
\frac{p(x^*)}{2}(x^Tx^*)^d-th_d(x;x^*)\ge 0\quad\Rightarrow\quad p(x)
\ge 0.
\end{equation}
To gain geometric intuition, consider the region defined by
\begin{equation}\label{eq:spherical_cone}
A(x^Tx^*)^d-th_d(x;x^*)\ge 0,
\end{equation}
where $A$ and $t$ are positive scalars. Let $\theta$ denote the angle between $x$ and $x^*$. Since   $x^Tx^*=\|x\|_2\cos\theta$, the inequality above is equivalent to 
$\sum_{k=1}^{d/2}(\tan\theta)^{2k}\le A/t$.
Therefore, the subregion corresponds to $|\theta|\le\arctan z$, where $z$ is the unique positive real root of 
$\sum_{k=1}^{d/2} z^{2k}= A/t$. Consequently, identity~\eqref{eq:degd_homo} shows that a sufficiently small angle between $x$ and $x^*$ implies nonnegativity of $p(x)$.
\textcolor{black}{Figure~\ref{fig:spherical_cap} illustrates the intersection of the region defined by~\eqref{eq:spherical_cone} and $\mbS^{n-1}$, which is a spherical cap.}
\end{remark}

\begin{figure}[ht]
    \centering
    \includegraphics[width=0.3\textwidth]{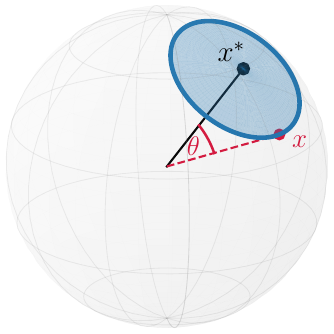}
    \caption{\textcolor{black}{A spherical cap centered at $x^*$: the local identity~\eqref{eq:degd_homo} certifies $p(x)\ge 0$ for every $x\in \mbS^{n-1}$ within angle $\theta$ of $x^*$.}}
    \label{fig:spherical_cap}
\end{figure}


To prove Theorem~\ref{thm:degd_homo}, we first establish the following lemma. 
Let us recall the multi-index notation that will be used in its proof. For an $n$-tuple of nonnegative integers $\alpha\defeq(\alpha_1,\ldots,\alpha_n)$, and a vector $x\in\reals^n$, we write $x^{\alpha}$ to denote $x_1^{\alpha_1}\ldots x_n^{\alpha_n}$ and $|\alpha|$ to denote $\sum_{i=1}^n \alpha_i$. For two multi-indices $\alpha$ and $\beta$, we denote by $\alpha+\beta$ their componentwise sum.


\begin{lemma}\label{thm:degd_nonhomo}
Consider a (not necessarily homogeneous) polynomial~$p$ in $n$ variables and degree at most $d$, where $d$ is even. If $p(x^*)>0$ for some $x^*\in\reals^n$, then for $t>0$ sufficiently large, $p(x)$ can be written as
\[p(x)=\left(\frac{p(x^*)}{2}-tg_d(x-x^*)\right)+\sigma(x),\]
where $\sigma(x)\in\SOS_{n,d}$ and 
\begin{equation}\label{eq:degd_nonhomo_g}
  g_d(x)\defeq\sum_{k=1}^{d/2}\|x\|_2^{2k}.  
\end{equation}
In particular, the above statement holds for any 
\[t\ge \tilde{T}(p;x^*)\defeq\dfrac{n}{2p(x^*)}\|q\|^2_{\infty} + \|q\|_1,\]
where $q(x)=p(x+x^*)-p(x^*)$. 
\end{lemma}

\begin{proof}
Without loss of generality, we assume $x^*=0$; otherwise, applying the statement of the lemma to the polynomial $p(x+x^*)$ and recalling that the existence of an sos decomposition is invariant under an affine change of variables, yields the desired result. 
Let $$r_t(x)\defeq p(x)+tg_d(x)-\frac{p(0)}2=q(x)+tg_d(x)+\frac{p(0)}2.$$ 
Thus, our task is to show that $r_t(x)$ is sos for $t\ge \tilde{T}(p;0)$. 
To do this, we work with the following representation of the summands: 
\[\begin{aligned}
    q(x)+\frac{p(0)}{2}&=z(x)^TQz(x)=\sum_{\alpha:|\alpha|\le d/2}\sum_{\beta:|\beta|\le d/2} Q_{\alpha,\beta}x^{\alpha+\beta},\\
    g_d(x)&=z(x)^TDz(x)=\sum_{\alpha:|\alpha|\le d/2}\sum_{\beta:|\beta|\le d/2} D_{\alpha,\beta}x^{\alpha+\beta}.
\end{aligned}\]
Here, $z(x)\defeq(1,x_1,\ldots,x_n,\ldots,x_1^{d/2},\ldots,x_n^{d/2})$ is the vector of monomials of degree at most $d/2$, and $Q$ and $D$ are symmetric matrices. We use the multi-index notation 
$Q_{\alpha,\beta}$ (resp. $D_{\alpha,\beta}$) to refer to the appropriate entry of the matrix $Q$ (resp. $D$). 


Since $g_d(0)=0$, we have $D_{0,0}=0$. Observe that $g_d(x)$ is a positive linear combination of all monomials of the form $x^{2\alpha}$, where the multi-index $\alpha\neq 0$ and $|\alpha|\le d/2$. Thus, we can take $D$ to be diagonal with $D_{\alpha,\alpha}\ge 1$ for every nonzero multi-index $\alpha$. For the matrix $Q$, note that $q(0)=0$, so  $Q_{0,0}=p(0)/2$. Furthermore, the entries of $Q$ can be rearranged such that $Q_{0,\alpha}=Q_{\alpha,0}=0$ for all $\alpha$ with $|\alpha|\in\{2,\ldots, d/2\}$. This is possible because each monomial $x^{\alpha}$ with $|\alpha|\ge 2$ can be written as the product $x^{\beta}x^{\gamma}$ for some nonzero multi-indices $\beta,\gamma$ with $\beta+\gamma=\alpha$. With this rearrangement, we can write
\begin{equation}\label{eq:degd_nonhomo_rtx}
    \begin{aligned}
r_t(x)&=z(x)^T(Q+tD)z(x)\\
&=\frac{p(0)}{2n}\sum_{i=1}^n\left(1+\frac{nq_i}{p(0)}x_i\right)^2+z(x)^T(\hat{Q}+tD)z(x),
\end{aligned}
\end{equation}
where $q_i$ denotes the coefficient of the term $x_i$ in $q(x)$, and $\hat{Q}$ is defined entrywise as 
\[\hat{Q}_{\alpha,\beta}=
\begin{cases}
    0  & \text{if } \alpha= 0 \text{ or } \beta=0, \\
    Q_{\alpha,\beta}-\dfrac{nq_{i}^2}{2p(0)} & \text{if } (\alpha,\beta)=(e_i,e_i),\quad i=1,\ldots,n,\\ 
    Q_{\alpha,\beta} & \text{otherwise.}
\end{cases}\]
Here, $e_i$ is the $i$-th standard basis vector of $\reals^n$ (so that $x^{e_i}=x_i$ in multi-index notation). Finally, one can verify that when $t\ge \tilde{T}(p;0)$, the matrix $\hat{Q}+tD$ is diagonally dominant\footnote{Recall that a symmetric matrix $A=(a_{ij})$ is diagonally dominant if $a_{ii}\ge \sum_{j\neq i}|a_{ij}|$ for all $i$. It follows from Gershgorin's circle theorem~\cite{gershgorin1931uber} that diagonally {dominant} matrices are positive semidefinite.} and thus positive semidefinite. Therefore, when $t\ge\tilde{T}(p;0)$, the polynomial $z(x)^T(\hat{Q}+tD)z(x)$ is sos. In view of~\eqref{eq:degd_nonhomo_rtx}, when $t\ge\tilde{T}(p;0)$, $r_t(x)$ is also sos.

\end{proof}
We are now ready to prove Theorem~\ref{thm:degd_homo}.



\begin{proof}[Proof of Theorem~\ref{thm:degd_homo}]
We first consider the special case where $x^*=e_1$. Let $z\defeq(x_2,\ldots,x_n)$. We dehomogenize~$p$ by defining $\tilde{p}(z)\defeq p(1,z)$. Then $\tilde{p}(z)$ is a polynomial with $\tilde{p}(0)=p(e_1)>0$. 
By Lemma~\ref{thm:degd_nonhomo}, whenever $t\ge \tilde{T}(\tilde{p};0)$, $\tilde{p}(z)$ admits the representation
\[\tilde{p}(z)=\left(\frac{p(e_1)}{2}-tg_d(z)\right)+\tilde{\sigma}(z),\]
where $\tilde{\sigma}(z)\in\SOS_{n-1,d}$ and $g_d$ is defined as in~\eqref{eq:degd_nonhomo_g}. 
Next, homogenizing $\tilde{p}(z)$ yields the identity:
\[\begin{aligned}
p(x)&=p(x_1,z)=x_1^d \tilde{p}(z/x_1)\\
&=\left(\frac{p(e_1)}{2}x_1^d-t\sum\limits_{k=1}^{d/2}\|z\|_2^{2k}x_1^{d-2k}\right)+x_1^d \tilde{\sigma}(z/x_1)\\
&=\left(\frac{p(e_1)}{2}x_1^d-t\sum\limits_{k=1}^{d/2}(\|x\|_2^2-x_1^2)^{k}x_1^{d-2k}\right)+\sigma(x)\\
&=\left(\frac{p(e_1)}{2}x_1^d-th_d(x;e_1)\right)+\sigma(x),
\end{aligned}\]
where $\sigma(x)=x_1^d \tilde{\sigma}(z/x_1)\in\HSOS_{n,d}$, since homogenization preserves the sos property. Note that
\[\begin{aligned}
T(p;e_1)&\ge \dfrac{n}{2p(e_1)}\|p\|_{\infty}^2+\|p\|_1\\
&= \dfrac{n}{2\tilde{p}(0)}\|\tilde{p}\|^2_{\infty} + \|\tilde{p}\|_1\\
&=\tilde{T}(\tilde{p};0),
\end{aligned}\]
where the second line follows from the fact that the coefficients of~$p$ remain the same after dehomogenization. 
Therefore, in view of the fact that $h_d(x;e_1)$ is sos, the claim of Theorem~\ref{thm:degd_homo} follows in the case where $x^*=e_1$.


To establish the claim for a general point $x^*\in \mbS^{n-1}$, observe that for any such point there exists a matrix $U\in\reals^{n\times n}$ with $U^TU=I$, such that~$Ue_1=x^*$. Define $p_U(x)\defeq p(Ux)$. Then $p_U$ is a polynomial with $p_U(e_1)=p(x^*)>0$. Applying the statement of the theorem to~$p_U$, whenever $t\ge T(p_U;e_1)$, we have
\[
p_U(x)=\left(\frac{p_U(e_1)}{2}(x^Te_1)^d-t\sum\limits_{k=1}^{d/2}(\|x\|_2^2-(x^Te_1)^2)^{k}(x^Te_1)^{d-2k}\right)+\hat{\sigma}(x),\]
where $\hat{\sigma}(x)\in\HSOS_{n,d}$.
Substituting $x$ with $U^Tx$ gives
\[\begin{aligned}
p(x)&=p_U(U^Tx)\\
&=\left(\frac{p(x^*)}{2}(x^TUe_1)^d-t\sum\limits_{k=1}^{d/2}(\|U^Tx\|_2^2-(x^TUe_1)^2)^{k}(x^TU  e_1)^{d-2k}\right)+\hat{\sigma}(U^Tx)\\
&=\left(\frac{p(x^*)}{2}(x^Tx^*)^d-t\sum\limits_{k=1}^{d/2}(\|x\|_2^2-(x^Tx^*)^2)^{k}(x^Tx^*)^{d-2k}\right)+\hat{\sigma}(U^Tx)\\
&=\left(\frac{p(x^*)}{2}(x^Tx^*)^d-th_d(x;x^*)\right)+\hat{\sigma}(U^Tx),
\end{aligned}
\]
where in the second line we use the fact that the 2-norm is unitarily invariant. 
Since ${\hat{\sigma}(x)\in\HSOS_{n,d}}$ and the sos property is preserved under a linear change of variables, we have ${\hat{\sigma}(U^Tx)\in\HSOS_{n,d}}$. Finally, note that $T(p_U;e_1)=T(p;x^*)$, since the set of orthogonal matrices is invariant under multiplication by any of its elements.
\end{proof}


\begin{remark}
A closer examination of the proof of Theorem~\ref{thm:degd_homo} reveals that the polynomial $\sigma(x)$ in identity~\eqref{eq:degd_homo} is not only sos, but in fact ``scaled-diagonally-dominant-sum-of-squares'' (sdsos) with respect to a specific basis~\cite{ahmadi2017sum}. 
Indeed, from equation~\eqref{eq:degd_nonhomo_rtx} in the proof of Lemma~\ref{thm:degd_nonhomo}, we see that the polynomial $\hat{\sigma}(x)$ is sdsos. Consequently, $\sigma(x)=\hat{\sigma}(U^Tx)$ is sdsos with respect to a rotated monomial basis induced by the linear transformation $U^T$. This implies that the search for an algebraic proof of polynomial nonnegativity in the form of identity~\eqref{eq:degd_homo} can be achieved by a second-order cone program (SOCP). This observation connects to Algorithm~\ref{alg:uniform_grid} in Section~\ref{sec:disos_grid}, which can potentially be implemented as an SOCP-based method.
\end{remark}


We now present another local Positivstellensatz, analogous to Theorem~\ref{thm:degd_homo}, which provides an algebraic certificate of nonnegativity for a polynomial~$p$ in a neighborhood of a point $x^*$ where $p$ is positive. The neighborhood is defined by a single polynomial of even degree $d' \in \{2,\ldots,d-2\}$. This formulation offers additional flexibility in the choice of the sos polynomial multiplied by the polynomial defining the neighborhood.



\begin{theorem}\label{thm:degd_homo_d'}
Let $d$ be an even integer, $d'$ be an even integer in $\{2,\ldots,d-2\}$, and $p\in H_{n,d}$. 
If $p(x^*)>0$ for some $x^*\in \mbS^{n-1}$, then for $t>0$ sufficiently large, $p(x)$ can be written as
\begin{equation*}
    p(x)=\left(\frac{p(x^*)}{2}(x^Tx^*)^{d'}-th_{d'}(x;x^*)\right)\left((x^Tx^*)^{d-d'}+h_{d-d'}(x;x^*)\right)+\sigma(x),
\end{equation*}
where $\sigma(x)\in\HSOS_{n,d}$, and $h_{d'}(x;x^*)$, $h_{d-d'}(x;x^*)$ are as in~\eqref{eq:degd_homo_h}. In particular, the above statement holds for any
\[t\ge T'(p;x^*)\defeq T(p;x^*) + \frac{p(x^*)}{2},\]
where $T(p;x^*)$ is defined as in~\eqref{eq:degd_homo_t}. 
\end{theorem}








The main step in the proof of Theorem~\ref{thm:degd_homo_d'} is the following lemma, which builds on Lemma~\ref{thm:degd_nonhomo}.


\begin{lemma}\label{thm:degd_nonhomo_d'}
Consider a (not necessarily homogeneous) polynomial~$p$ in $n$ variables and degree at most $d$, where $d$ is even. Let $d'$ be an even integer in $\{2,\ldots,d-2\}$. If $p(x^*)>0$ for some $x^*\in\reals^n$, then for $\hat{t}>0$ sufficiently large, $p(x)$ can be written as
\[p(x)=\left(\frac{p(x^*)}{2}-\hat{t}g_{d'}(x-x^*)\right)\left(1+g_{d-d'}(x-x^*)\right)+\hat{\sigma}(x),\]
where $\hat{\sigma}(x)\in\SOS_{n,d}$. In particular, the above statement holds for any 
\[\hat{t}\ge \tilde{T}(p;x^*)+\frac{p(x^*)}{2}.\]
Here, $g_{d'}(x), g_{d-d'}(x)$, and $\tilde{T}(p;x^*)$ are defined as in Lemma~\ref{thm:degd_nonhomo}.
\end{lemma}

\begin{proof}
Without loss of generality, we assume $x^*=0$; otherwise, applying the statement of the lemma to the polynomial $p(x+x^*)$ and recalling that the existence of an sos decomposition is invariant under an affine change of variables, yields the desired result. By Lemma~\ref{thm:degd_nonhomo}, for any $t\ge \tilde{T}(p;0)$, $p(x)$ admits the representation
\[p(x)=\frac{p(0)}{2}-tg_d(x)+\sigma(x),\]
where $\sigma(x)\in\SOS_{n,d}$. Let $A=p(0)/2$ and $\hat{t}=t+A$. Then 
\[\begin{aligned}
p(x)&=A-tg_d(x)+\sigma(x)\\
&=\left(A-\hat{t}g_{d'}(x)\right)(1+g_{d-d'}(x))-Ag_{d-d'}(x)+\hat{t}g_{d'}(x)(1+g_{d-d'}(x))-tg_d(x)+\sigma(x)\\
&=\left(A-\hat{t}g_{d'}(x)\right)(1+g_{d-d'}(x)) + A\left(g_{d}(x)-g_{d-d'}(x)\right)\\
&\qquad +\hat{t}\big(g_{d'}(x)(1+g_{d-d'}(x))-g_d(x)\big)+ \sigma(x)\\
&=\left(A-\hat{t}g_{d'}(x)\right)(1+g_{d-d'}(x))+\hat{\sigma}(x),
\end{aligned}\]
where \[\hat{\sigma}(x)=\sigma(x)+A\left(g_{d}(x)-g_{d-d'}(x)\right)+\hat{t}\big(g_{d'}(x)(1+g_{d-d'}(x))-g_d(x)\big).\]
It remains to show that $\hat{\sigma}(x)\in\SOS_{n,d}$.
Observe that
\[g_{d}(x)-g_{d-d'}(x)=\sum_{k=(d-d')/2+1}^{d/2}\|x\|_2^{2k},\]
and
\[\begin{aligned}
   g_{d'}(x)(1+g_{d-d'}(x))-g_d(x)&=g_{d'}(x)+\left(g_{d'-2}(x)+\|x\|_2^{d'}\right)g_{d-d'}(x)-g_d(x)\\
   &=g_{d'-2}(x)g_{d-d'}(x).
\end{aligned}\]
Hence, both expressions are sos, which implies $\hat{\sigma}(x)$ is sos.
\end{proof}

The proof of Theorem~\ref{thm:degd_homo_d'} follows the same approach as the proof of Theorem~\ref{thm:degd_homo} and is therefore only outlined. We first consider the special case where $x^* = e_1$. We dehomogenize~$p$ by defining $\tilde{p}(z)\defeq p(1,z)$, apply Lemma~\ref{thm:degd_nonhomo_d'} to $\tilde{p}$, and homogenize back to show that whenever  $t \geq T'(p; e_1)$, $p(x)$ can be written as 
\[
    p(x)=\left(\frac{p(e_1)}{2}(x^Te_1)^{d'}-th_{d'}(x;e_1)\right)\left((x^Te_1)^{d-d'}+h_{d-d'}(x;e_1)\right)+\sigma(x),
\]
where $\sigma(x) \in \HSOS_{n,d}$. We then establish the claim for a general point $x^* \in \mbS^{n-1}$ via an orthogonal change of variables, as in the proof of Theorem~\ref{thm:degd_homo}.

\subsection{From Local Positivstellens\"atze to a Global Positivstellensatz}\label{sec:global_psatz}


Building on the results established in the previous subsection, we now present the proof of Theorem~\ref{cor:strictly_pos}.



\begin{proof}[Proof of Theorem~\ref{cor:strictly_pos}]
For any integer $m\ge 2$, we
construct a net $\mcN_m=\{x^1,\ldots,x^{r_m}\}\subset\mbS^{n-1}$ such that for any $y\in \mbS^{n-1}$, there exists $x\in \mcN_m$ with $\langle x,y\rangle\ge m/\sqrt{m^2+1}$. 
Fix an even degree $d'\in\{2,\ldots,d\}$ and let the initial algebraic disjunction be $\mcD_{d'}^1=\{\{0\}\}$. For $m\ge 2$, we construct the subsequent algebraic disjunctions as 
\begin{equation}\label{eq:strictly_pos_ad}
 \mcD^{m}_{d'}=\big\{\{(x^Tx^i)^{d'}-t_mh_{d'}(x;x^i)\},\quad i=1,\ldots,r_m\big\},
\end{equation}
where $t_m=\left(\sum_{k=1}^{d'/2}m^{-2k}\right)^{-1}$ and $h_{d}(x;x^i)$ is defined as in~\eqref{eq:degd_homo_h}. 
For any $x\in\mbS^{n-1}$ and any $x^i\in\mcN_m$, let $\theta\in[0,\pi]$ be the unique angle satisfying $\langle x,x^i\rangle=\cos\theta$. Then, it follows that
\[\begin{aligned}
    (x^Tx^i)^{d'}-t_mh_{d'}(x;x^i)&=(\cos\theta)^{d'}-t_m\sum_{k=1}^{d'/2}(\sin\theta)^{2k}(\cos\theta)^{d'-2k}\\
    &=(\cos\theta)^{d'}\left(1-t_m\sum_{k=1}^{d'/2}(\tan\theta)^{2k}\right).
\end{aligned}\]
Therefore, the inequality $(x^Tx^i)^{d'}-t_mh_{d'}(x;x^i)\ge 0$ holds if and only if 
\[
  \sum_{k=1}^{d'/2}(\tan\theta)^{2k}\le t_m^{-1}
  \quad\Leftrightarrow\quad |\tan\theta|\le m^{-1} \quad
  \Leftrightarrow\quad\langle x,x^i\rangle\ge \frac{m}{\sqrt{m^2+1}}. 
\]
Hence $\mcD_{d'}^{m}$ forms a valid algebraic disjunction of $\reals^n$ by the construction of the net. 

Finally, note that $t_m\to\infty$ as $m\to\infty$. 
Therefore, keeping in mind that $\mbS^{n-1}$ is a compact set, for any positive definite form $p\in H_{n,d}$, letting $p_{\min}\defeq \min_{x\in\mathbf{S}^{n-1}}p(x)$,  
there exists $m$ such that $t_m$ exceeds $2T(p;x^*)/p_{\min}$ (resp. $2T'(p;x^*)/p_{\min}$) for any $x^*\in\mbS^{n-1}$, where $T(p;x^*)$ (resp. $T'(p;x^*)$) is defined in Theorem~\ref{thm:degd_homo} (resp. Theorem~\ref{thm:degd_homo_d'}). 
This guarantees that $p\in \DiSOS_{n,d}^{d}(\mcD_{d'}^{m})$ {for} any fixed even degree $d'\in\{2,\ldots,d\}$.
\end{proof}

The preceding proof in fact yields an explicit bound on the number of subregions required to certify nonnegativity of a polynomial with a disjunctive sos certificate. As an illustration, we present such a bound for the case $d'=d$ in Theorem~\ref{thm:num_ad} below. Given the NP-hardness of certifying polynomial nonnegativity, any such bound is expected to grow exponentially with the dimension.

\begin{theorem}\label{thm:num_ad}
Consider a positive definite form $p\in H_{n,d}$. Let $p_{\min}\defeq\min_{x\in\mbS^{n-1}} p(x)$, and define
 \begin{equation}\label{eq:thm_counting}
T^*(p)\defeq \sup_{U\in\reals^{n\times n}:U^TU=I_n}\left( \frac{n}{2p_{\min}}\|p(Ux)\|_{\infty}^2 + \|p(Ux)\|_1\right),\quad
M(p)\defeq \left\lceil \sqrt{\frac{T^*(p)d}{p_{\min}}}\;\right\rceil.
\end{equation}
Let $\mcD_d^m$ be the family of algebraic disjunctions constructed in~\eqref{eq:strictly_pos_ad}, where each subset of $\mcD_d^m$ contains exactly one degree-$d$ form. Then $p\in\DiSOS_{n,d}^d(\mcD_d^m)$ for all $m\ge M(p)$. In particular, $\left(1+2 M(p)\right)^n$ subregions, each defined by a single degree-$d$ form, suffice to establish a disjunctive sos proof of nonnegativity of~$p$.
\end{theorem}

\begin{proof}
Recall the definition of the algebraic disjunction
\[
 \mcD_d^{m}=\big\{\{(x^Tx^i)^{d}-t_mh_{d}(x;x^i)\},\quad i=1,\ldots,r_m\big\},   
\]
where $t_m=\left(\sum_{k=1}^{d/2}m^{-2k}\right)^{-1}$. For each $i=1,\ldots,r_m$, by Theorem~\ref{thm:degd_homo}, there exists a polynomial $\sigma_i\in\HSOS_{n,d}$ such that
 \[    p(x)=\left(\frac{p(x^i)}{2}(x^Tx^i)^d-T(p;x^i)h_d(x;x^i)\right)+\sigma_i(x),\]
 where $T(p;\cdot)$ is defined as in~\eqref{eq:degd_homo_t}. 
Rewriting gives
  \[\begin{aligned}    
  p(x)&=\frac{p(x^i)}{2}(x^Tx^i)^d+\sigma_i(x)-T(p;x^i)h_d(x;x^i)\\
  &=\frac{p(x^i)}{2}\left((x^Tx^i)^d-t_mh_d(x;x^i)\right)+\left(\sigma_i(x)+\left(\frac{p(x^i)}{2}t_m-T(p;x^i)\right)h_d(x;x^i)\right).
  \end{aligned}\]
   Note that $t_m\ge 2m^2/d$. Therefore, whenever $m\ge M(p)$,
  \[\frac{p(x^i)}{2}t_m\ge \frac{p_{\min}}{2}\cdot\frac{2M^2(p)}{d}\ge T^*(p)\ge T(p;x^i),\quad i=1,\ldots,r_m.\]
 Since $h_d(x;x^*)\in\HSOS_{n,d}$ for any $x^*\in\reals^n$, it follows that
 \[\sigma_i(x)+\left(\frac{p(x^i)}{2}t_m-T(p;x^i)\right)h_d(x;x^i)\in\HSOS_{n,d}.\]
 Thus, $p\in\DiSOS_{n,d}^d(\mcD_d^m)$ whenever $m\ge M(p)$. We now bound the size of $\mcN_m$ when $m=M(p)$. Recall that $\mcN_m=\{x^1,\ldots,x^{r_m}\}\subset\mbS^{n-1}$ is a net such that for any $y\in \mbS^{n-1}$, there exists $x\in \mcN_m$ with $\langle x,y\rangle\ge m/\sqrt{m^2+1}$. This implies that
 \[\begin{aligned}
\|x-y\|_2^2&=2-2\langle x,y\rangle\le 2\left(1-\frac{m}{\sqrt{m^2+1}}\right)\\
&= \frac{2(\sqrt{m^2+1}-m)}{\sqrt{m^2+1}}=\frac{2}{(\sqrt{m^2+1}+m)\sqrt{m^2+1}}\\
&\le \frac{1}{m^2}.
 \end{aligned}\]
By \cite[Corollary 4.2.11]{vershynin2018high}, we conclude that 
$r_m\le (1+2m)^n= (1+2M(p))^n$.
\end{proof}

\subsection{An Algorithm Based on Spherical-Cap Subregions} \label{sec:disos_grid}
Building on Theorem~\ref{cor:strictly_pos}, we propose an algorithm to approximate the minimum value $p_{\min}$ of a form $p\in H_{n,d}$ over the unit sphere $\mbS^{n-1}$. This quantity is also equal to the optimal value of the following optimization problem:
\begin{equation}\label{eq:poly_min}
	p_{\min} =\; \begin{array}[t]{cl}
		\max\limits_{\gamma\in{\scriptsize\reals}} & \gamma\\
		\st & p(x)-\gamma\|x\|_2^d\ge 0,\quad\forall x\in\reals^n.
	\end{array}
\end{equation}
For any even degree $d'\in\{2,\ldots,d\}$, we can present a complete SDP-based hierarchy of lower bounds on $p_{\min}$. At level $m$ of this hierarchy, we construct a net $\mcN_m$ of the unit sphere\footnote{There are different techniques for constructing such a net; see, e.g.,~\cite{leopardi_partition_2006,chakraborty2022generating}.} and the corresponding algebraic disjunction $\mcD_{d'}^m$ defined in~\eqref{eq:strictly_pos_ad}. Using the notation introduced in Definition~\ref{def:disos}, we compute a lower bound $p^{(m)}$ by solving the following SDP:
\begin{equation}\label{eq:poly_min_disos_relax}
	p^{(m)} \defeq \begin{array}[t]{cl}
		\max\limits_{\gamma\in{\scriptsize\reals}} & \gamma\\
		\st & p(x)-\gamma\|x\|_2^d\in\DiSOS_{n,d}^{d}(\mcD_{d'}^m).
	\end{array}
\end{equation}
{This SDP can be decomposed into $r_m$ independent SDPs, one for each subset in $\mcD_{d'}^m$, and solved in parallel; see Remark~\ref{remark:uniform_grid} for details.} 
The next theorem utilizes Theorem~\ref{thm:num_ad} to establish a convergence rate for the sequence of lower bounds $p^{(m)}$ produced by the SDP hierarchy in~\eqref{eq:poly_min_disos_relax} (in the case $d'=d$ as a representative example).


\begin{theorem}\label{thm:poly_min_disos_lb}
Let~$p\in H_{n,d}$ and~$p_{\min}\defeq\min_{x\in\mbS^{n-1}} p(x)$. Let $\mcD^m_d$ be the family of algebraic disjunctions constructed in~\eqref{eq:strictly_pos_ad}, where each subset of $\mcD_d^m$ contains exactly one degree-$d$ form. 
Then, we have
\[p_{\min}-p^{(m)}= O(1/m).\] 
\end{theorem}
\begin{proof}
Let $q(x)\defeq\|x\|_2^d$ and fix $\epsilon\in(0,p_{\min})$. By Theorem~\ref{thm:num_ad}, we have $p-\epsilon q\in\DiSOS_{n,d}^d(\mcD_d^m)$ for all $m\ge M(p-\epsilon q)$, where $M(\cdot)$ is defined in~\eqref{eq:thm_counting}. Thus, $ p^{(m)}\ge \epsilon$ 
whenever $m\ge M(p-\epsilon q)$. Let $c_{\infty}\defeq \|q\|_{\infty},c_1\defeq\|q\|_1$, which are constants depending only on $n$ and $d$. We now bound $T^*(p - \epsilon q)$ from above using the definition in~\eqref{eq:thm_counting}:
\[\begin{aligned}
T^*(p-\epsilon q)&= \sup_{U\in\reals^{n\times n}:U^TU=I_n}\left( \frac{n}{2(p_{\min}-\epsilon)}\|(p-\epsilon q)(Ux)\|_{\infty}^2 + \|(p-\epsilon q)(Ux)\|_1\right) \\
&\le \sup_{U\in\reals^{n\times n}:U^TU=I_n}\left( \frac{n}{p_{\min}-\epsilon}\left(\|p(Ux)\|_{\infty}^2+\epsilon^2\|q(Ux)\|_{\infty}^2\right) + \|p(Ux)\|_1+\epsilon \|q(Ux)\|_1\right)\\
&= \sup_{U\in\reals^{n\times n}:U^TU=I_n}\left( \frac{n}{p_{\min}-\epsilon}\left(\|p(Ux)\|_{\infty}^2+c_{\infty}^2\epsilon^2\right) + \|p(Ux)\|_1+c_1\epsilon\right)\\
&\le \frac{2p_{\min}}{p_{\min}-\epsilon} T^*(p)+\frac{c_{\infty}^2n\epsilon^2}{p_{\min}-\epsilon}+c_1\epsilon\\
&= \frac{2p_{\min}T^*(p)+c_{\infty}^2n\epsilon^2+c_1\epsilon(p_{\min}-\epsilon)}{p_{\min}-\epsilon}\\
&\le \frac{2p_{\min}T^*(p)+(c_{\infty}^2n+c_1/4)p_{\min}^2}{p_{\min}-\epsilon}.
\end{aligned}\]
Here, in the first inequality we used the triangle inequality for the polynomial norms. 
Therefore, for any $\epsilon\in(0,p_{\min})$, 
\[M(p-\epsilon q)\le \sqrt{\frac{2T^*(p-\epsilon q)d}{p_{\min}-\epsilon}}+1\le \frac{\sqrt{2(2p_{\min}T^*(p)+(c_{\infty}^2n+c_1/4)p_{\min}^2)d}}{p_{\min}-\epsilon}+1.\]
Finally, since the above bound holds for every $\epsilon\in(0,p_{\min})$, when~$m$ is sufficiently large, we may choose $\epsilon\in(0,p_{\min})$ such that
\[m = \frac{\sqrt{2(2p_{\min}T^*(p)+(c_{\infty}^2n+c_1/4)p_{\min}^2)d}}{p_{\min}-\epsilon}+1.\]
With this choice, it follows that
\[p_{\min}-p^{(m)}\le p_{\min}-\epsilon= \frac{\sqrt{2(2p_{\min}T^*(p)+(c_{\infty}^2n+c_1/4)p_{\min}^2)d}}{m-1}=O(1/m).\]
\end{proof}

Since the net $\mcN_m$ already appears in the formulation of our SDP, it is natural to use the points in $\mcN_m$ to obtain upper bounds on $p_{\min}$ via direct point evaluation. 
As the level $m$ increases, the resolution of the net improves and its covering radius tends to zero, yielding increasingly tight upper bounds on $p_{\min}$. In particular, by adapting the argument in the comments following~\cite[Proposition 4]{vianello_subperiodic_2018}, one can construct a net such that the error between the resulting upper bound at level $m$ and $p_{\min}$ is of order $O(1/m^2)$. Combined with Theorem~\ref{thm:poly_min_disos_lb}, this implies that the gap between the lower and upper bounds on $p_{\min}$ converges to zero as $m$ increases, with an overall rate of $O(1/m)$. Therefore, for any prescribed tolerance $\epsilon_{\text{tol}}>0$ to approximate $p_{\min}$, Algorithm~\ref{alg:uniform_grid} is guaranteed to terminate within $O(1/\epsilon_{\text{tol}})$ iterations. 



\begin{algorithm}[!tb]
\caption{Spherical-Cap-Based Algorithm for Minimizing a Form}
\label{alg:uniform_grid}

\newlength{\algmathwidth}
\setlength{\algmathwidth}{\dimexpr\linewidth-\algorithmicindent\relax}
\begin{algorithmic}
\Require A polynomial $p\in H_{n,d}$, an even number $d'\in\{2,\ldots,d\}$, a tolerance $\epsilon_{\text{tol}}>0$.
\Ensure  A lower bound $L$ on the minimum value of $p$ over $\mbS^{n-1}$ and a point $\bar{x}\in\mbS^{n-1}$ such that 
$p(\bar{x})-L\le \epsilon_{\text{tol}}$.
\Initialization Let $L=-\infty, U=\infty, m=1$.
\While{$U-L> \epsilon_{\text{tol}}$}
\State Generate a net $\mcN_m=\{x^1,\ldots,x^{r_m}\}\subset\mbS^{n-1}$ such that for all $y\in \mbS^{n-1}$, there exists\\
\hspace*{\algorithmicindent} $x\in \mcN_m$ with $\langle x,y\rangle\ge m/\sqrt{m^2+1}$.
\State $t_m\gets\left(\sum_{k=1}^{d'/2}m^{-2k}\right)^{-1}$.
\State 
$L^m\gets$ \parbox[t]{\algmathwidth}{%
$\begin{array}[t]{cl}
 \max\limits_{\gamma,s_i,\hat{s}_i} & \gamma\\
 \st & p(x)-\gamma\|x\|_2^d= s_{i}(x)+\left((x^Tx^i)^{d'}-t_mh_{d'}(x;x^i)\right)\hat{s}_{i}(x),\quad i=1,\ldots,r_m\\
 & s_{i}\in\HSOS_{n,d},\hat{s}_{i}\in\HSOS_{n,d-d'},\quad i=1,\ldots,r_m.
\end{array}$
}
\State $L\gets\max\{L,L^m\}$.
\State $U^m\gets\min_{1\le i\le r_m}\{p(x^i)\}$, $j\gets\argmin_{1\le i\le r_m}\{p(x^i)\}$
\Comment{Break ties arbitrarily.}
\If{$U^m < U$} $U\gets U^m$, $\bar{x}\gets x^{j}$.
\EndIf
\State $m\gets m+1$.
\EndWhile
\State \Return{$L, \bar{x}$}
\end{algorithmic}
\end{algorithm}

\begin{remark}\label{remark:uniform_grid}
{The computation of $L^m$ in Algorithm~\ref{alg:uniform_grid} can be decomposed into $r_m$ independent SDPs that can be solved in parallel.
Indeed, for each $k\in\{1,\ldots,r_m\}$, one can compute
\begin{equation*}
\gamma_k^*\defeq
\begin{array}[t]{cl}
\max\limits_{\gamma_k,s_k,\hat{s}_k} & \gamma_k\\
\st & p(x)-\gamma_k\|x\|_2^d= s_{k}(x)+\left((x^Tx^k)^{d'}-t_mh_{d'}(x;x^k)\right)\hat{s}_{k}(x)\\
& s_{k}\in\HSOS_{n,d},\ \hat{s}_{k}\in\HSOS_{n,d-d'},
\end{array}
\end{equation*}
and set $L^m\gets\min_{1\le k\le r_m}\gamma_k^*$.}
\end{remark}


In practice, it is often more efficient to refine the subregions of the unit sphere adaptively instead of increasing the resolution of the net uniformly. Such an adaptive algorithm can be implemented within a spatial branch-and-bound framework, as discussed in Section~\ref{sec:bnb}.

\section{An Optimization-Free Disjunctive Positivstellensatz}\label{sec:disos_new}
In this section, we introduce a new disjunctive approach to certify polynomial nonnegativity. Unlike our previous approach, which relies on sos polynomials and {solutions} to SDPs, the new approach only requires checking nonnegativity of the coefficients of some polynomials we can explicitly write down. In this framework, \emph{checking} polynomial nonnegativity does not require any optimization solver at all, and \emph{imposing} a nonnegativity constraint on a polynomial leads to a linear program. Our main result in this section is an optimization-free disjunctive Positivstellensatz (Theorem~\ref{cor:strictly_pos_new}), analogous to Theorem~\ref{cor:strictly_pos}. 

\subsection{Simplicial Disjunctions and Optimization-Free Certificates}
In this section, we work with algebraic disjunctions where the corresponding subregions of the space are polyhedral cones. 
Recall that a polyhedral cone in $\reals^n$ can be represented as a conic combination of a finite number of vectors $v_1,\ldots,v_{\ell}\in\reals^n$: 
\[\cone\{v_1,\ldots,v_{\ell}\}=\{x\in\reals^n\mid x=\alpha_1v_1+\ldots+\alpha_{\ell}v_{\ell}\;\text{ for some }\alpha_1,\ldots,\alpha_{\ell}\ge 0\}.\]
We say that a polyhedral cone is simplicial if it is generated by linearly independent vectors.

\begin{definition}
Let $r$ be a positive integer. We say that a collection of $n\times n$ invertible matrices $\mcV=\{V_k\}_{k=1,\ldots,r}$ forms a \emph{simplicial disjunction} of $\reals^n$ if 
\[\mathbb{R}^n=\bigcup\limits_{k=1}^r\Omega_k,\quad \text{where }\quad\Omega_k=\cone\{V_k\}\defeq\cone\{v_{k,1}, \ldots, v_{k,n}\},\]
and $v_{k,j}$ is the $j$-th column of $V_k$.
\end{definition}

Note that any vector $x\in\cone\{V_k\}$ can be written as $x=V_kz$ for some elementwise nonnegative vector $z\in\reals^n$. Since $V_k$ is invertible, this implies that $$\cone\{V_k\}=\{x\in\reals^n\mid V_k^{-1}x\ge 0\}.$$ 
Therefore, the simplicial disjunction $\mcV$ can be viewed as a special case of an algebraic disjunction introduced in Definition~\ref{def:disjunction} with $n_k=n$ for $k=1,\ldots,r$, and polynomials $q_{k,j}(x)=c_{k,j}^Tx$ for $k=1,\ldots,r,j=1,\ldots,n$, where $c_{k,j}\in\reals^n$ are the columns of $V_k^{-T}$. 
We now explain how simplicial disjunctions can help certify polynomial nonnegativity. Because of the homogenization argument given in Section~\ref{sec:local_psatz}, we restrict attention to forms.
\begin{definition}\label{def:disos_new}
A polynomial~$p\in H_{n,d}$ is said to \emph{have nonnegative coefficients with respect to a simplicial disjunction} $\mcV=\{V_k\}_{k=1,\ldots,r}$ if $p(V_kx)$ has nonnegative coefficients for all $k=1,\ldots, r$. We denote the set of such polynomials by $\DiSOSL_{n,d}(\mcV)$. 
\end{definition}
Analogous to the cone $\DiSOS_{n,d}^{\bar{d}}(\mcD)$ 
in Definition~\ref{def:disos}, the set $\DiSOSL_{n,d}(\mcV)$ is a convex cone.  
Moreover, any polynomial $p\in\DiSOSL_{n,d}(\mcV)$ has an explicit disjunctive proof of nonnegativity. Indeed, for $k=1,\ldots,r$, let $\sigma_k(x)\defeq p(V_k x)$, which, by definition, has nonnegative coefficients. We can therefore write 
\[p(x)=\sigma_k(V_k^{-1}x),\quad k=1,\ldots,r.\]
Since $\sigma_k$ has nonnegative coefficients, we have the implication
\[V_k^{-1}x\ge 0\ \Rightarrow\ \sigma_k(V_k^{-1}x)\ge 0,\]
for all $k=1,\ldots,r$. Equivalently, $x\in\cone\{V_k\}\ \Rightarrow\ p(x)\ge 0$. Since $\bigcup_{k=1}^r\cone\{V_k\}=\reals^n$ by the definition of a simplicial disjunction,~$p$ is nonnegative over $\reals^n$.

We now establish the following optimization-free disjunctive Positivstellensatz, which runs parallel to the disjunctive Positivstellensatz in Theorem~\ref{cor:strictly_pos}. It shows that every positive definite form has a disjunctive proof of nonnegativity of the above type.
\begin{theorem}[Optimization-Free Disjunctive Positivstellensatz]\label{cor:strictly_pos_new}
For any fixed dimension $n$ and degree $d$, there exists an explicit family of simplicial disjunctions
$\mcV^m$, indexed by $m$, such that for any positive
 definite form $p\in H_{n,d}$, we have $p\in \DiSOSL_{n,d}(\mcV^m)$ for some $m$.
\end{theorem}
\subsection{A Local Optimization-Free Positivstellensatz}
 To prove Theorem~\ref{cor:strictly_pos_new}, we first establish {a} local optimization-free Positivstellensatz in Theorem~\ref{thm:degd_homo_new}, which runs parallel to Theorem~\ref{thm:degd_homo}. Recall that $\|\cdot\|_{\infty}$ denotes the largest coefficient in absolute value of a polynomial in the standard monomial basis. For an $n$-dimensional multi-index $\alpha$ with $|\alpha|=d$, we define
 \[\binom{d}{\alpha}\defeq \frac{d!}{\alpha_1!\ldots\alpha_n!}.\]
We denote the $n$-dimensional vector of all ones by $\mathbf{1}$.


\begin{theorem}\label{thm:degd_homo_new}
For a polynomial $p\in H_{n,d}$ with $d$ even, if $p(x^*)>0$ for some $x^*\in \mbS^{n-1}$, then there exists $\delta >0$ such that for any matrix $V\in\reals^{n\times n}$ with $\|V-x^*\mathbf{1}^T\|_{F}\le \delta$, the polynomial $p(Vx)$ has nonnegative coefficients.
In particular, the above statement holds for any
\[\delta \le \delta(p;x^*)\defeq \sup_{\epsilon>1/p(x^*)}\min\left\{\frac{(p(x^*)\epsilon)^{1/d}-1}{( p(x^*)\epsilon)^{1/d}+1},\frac{1}{(1+\eta(p)\epsilon)n}\right\},\]
where
\[\eta(p)\defeq\sup_{U\in\reals^{n\times n}:U^TU=I_n} \|p(Ux)\|_{\infty}.\]
\end{theorem}

\begin{remark}
The polynomial $p(Vx)$ having nonnegative coefficients implies that $p(x)\ge 0$ for all $x\in\cone\{V\}$. In particular, when $V$ is invertible, $\cone\{V\}$ is simplicial and full-dimensional, which provides a proof of nonnegativity of~$p$ over a full-dimensional neighborhood around~$x^*$.
\end{remark}


\begin{proof}
We first consider the special case where $x^*=e_1$. Let $p(x)=\sum_{\alpha:|\alpha|=d}p_{\alpha}x^{\alpha}$. For each multi-index $\beta$ with $|\beta|=d$, we estimate the coefficient of the monomial $x^{\beta}$ in the expansion $p(Vx)=\sum_{\alpha:|\alpha|=d}p_{\alpha}(Vx)^{\alpha}$. 
Let $c(\alpha,\beta)$ denote the coefficient of $x^{\beta}$ arising from the expansion of $(Vx)^{\alpha}$. Then, the coefficient of $x^{\beta}$ in $p(Vx)$ is given by $$\sum_{\alpha:|\alpha|=d}p_{\alpha}c(\alpha,\beta).$$
In the case where $\alpha\neq de_1$ (i.e., $x^{\alpha}\neq x_1^d$), we bound $c(\alpha,\beta)$ by first noting that $$(Vx)^{\alpha}=\prod_{k=1}^n(V_{k1}x_1+\ldots+V_{kn}x_n)^{\alpha_k}.$$ 
Since $\|V-e_1\mathbf{1}^T\|_{F}\le \delta$, it follows that $V_{1j}\in[1-\delta,1+\delta]$ for $j=1,\ldots, n$, and $V_{kj}\in[-\delta,\delta]$ for $k=2,\ldots, n$ and $j=1,\ldots, n$. Hence, $\max_{1\le j\le n}|V_{1j}|\le 1+\delta$ and $\max_{1\le j\le n}|V_{kj}|\le \delta$ for  $k=2,\ldots,n$. 
Therefore, recalling the multinomial theorem, we have
\begin{equation}\label{eq:optfree_b1}
\begin{aligned}
|c(\alpha,\beta)|&\le \sum_{\substack{\gamma_1+\ldots+\gamma_n=\beta \\ |\gamma_i|=\alpha_i,\; i=1,\ldots,n}}\prod_{k=1}^n \left(\max_{1\le j\le n}|V_{kj}|\right)^{\alpha_k}\binom{\alpha_k}{\gamma_k}\\
&\le (1+\delta)^{\alpha_1}\delta^{d-\alpha_1}\sum_{\substack{\gamma_1+\ldots+\gamma_n=\beta \\ |\gamma_i|=\alpha_i,\; i=1,\ldots,n}}\prod_{k=1}^n \binom{\alpha_k}{\gamma_k}\\
&=(1+\delta)^{\alpha_1}\delta^{d-\alpha_1} \binom{d}{\beta},
\end{aligned}    
\end{equation}
The last equality follows from matching the coefficients of $x^{\beta}$ on both sides of the identity $(x_1+\dots+x_n)^d=\prod_{k=1}^n(x_1+\ldots+x_n)^{\alpha_k}$ and employing the multinomial theorem. 

We now turn to the case $\alpha=de_1$, where we have $(Vx)^{\alpha}=(V_{11}x_1+\ldots+V_{1n}x_n)^d$. For any $\delta\le \delta(p;e_1)$, we have $\delta < 1/n\le 1$. Thus, $\min_{1\le j\le n} V_{1j}\ge1-\delta\ge 0$, and
\begin{equation}\label{eq:optfree_b2}
c(\alpha,\beta)\ge \left(\min_{1\le j\le n}V_{1j}\right)^{d}\binom{d}{\beta}\ge (1-\delta)^d\binom{d}{\beta}.
\end{equation}
Combining the bounds~\eqref{eq:optfree_b1} for $\alpha\neq de_1$ and~\eqref{eq:optfree_b2} for $\alpha=de_1$, the coefficient of $x^{\beta}$ in $p(Vx)$ can be bounded from below as
\[\begin{aligned}
\sum_{\alpha:|\alpha|=d}p_{\alpha}c(\alpha,\beta)&=p_{de_1}c(de_1,\beta)+\sum_{\alpha:|\alpha|=d,
\alpha \neq de_1}p_{\alpha}c(\alpha,\beta)\\
&=p(e_1)c(de_1,\beta)+\sum_{l=0}^{d-1}\sum_{\alpha:|\alpha|=d,\alpha_1=l}p_{\alpha}c(\alpha,\beta)\\
&\ge \binom{d}{\beta}\left(p(e_1)(1-\delta)^d-\|p\|_{\infty}\sum_{l=0}^{d-1}\sum_{\alpha:|\alpha|=d,\alpha_1=l}(1+\delta)^{\alpha_1}\delta^{d-\alpha_1}\right)\\
&\ge\binom{d}{\beta}\left(p(e_1)(1-\delta)^d-\|p\|_{\infty}\sum_{l=0}^{d-1}n^{d-l}(1+\delta)^{l}\delta^{d-l}\right)\\
&=\binom{d}{\beta}\left(p(e_1)(1-\delta)^d-\|p\|_{\infty}\left(\frac{1+\delta}{n\delta}-1\right)^{-1}\left((1+\delta)^d-n^{d}\delta^d\right)\right)\\
&\ge\binom{d}{\beta}\left(p(e_1)(1-\delta)^d-\|p\|_{\infty}\left(\frac{1}{n\delta}-1\right)^{-1}(1+\delta)^d\right)\\
\end{aligned}\]
where the second inequality follows from the fact that the number of distinct $n$-dimensional multi-indices $\alpha$ satisfying $|\alpha|=d$ and $\alpha_1=l$ is given by $\binom{d-l+n-2}{n-2}\le n^{d-l}$. 

For any $\epsilon>1/p(e_1)>0$, since
$$\delta\le \frac{1}{(1+\eta(p)\epsilon)n}\le \frac{1}{(1+\|p\|_{\infty}\epsilon)n},$$
we have
\[\sum_{\alpha:|\alpha|=d}p_{\alpha}c(\alpha,\beta)\ge 
\binom{d}{\beta}\left(p(e_1)(1-\delta)^d-\epsilon^{-1}(1+\delta)^d\right).
\]
This expression is nonnegative because
$$\delta\le \frac{(p(e_1)\epsilon)^{1/d}-1}{(p(e_1)\epsilon)^{1/d}+1}.$$
Therefore, the coefficient of $x^{\beta}$ in $p(Vx)$ remains nonnegative for every multi-index $\beta$ with $|\beta|=d$ whenever $\delta\le\delta(p;e_1)$. 
Therefore, the theorem follows for the case where $x^*=e_1$.

For a general point $x^*\in \mbS^{n-1}$, there exists an orthogonal matrix $U\in\reals^{n\times n}$, such that $Ue_1=x^*$. Let  $p_U(x)=p(Ux)$. Then~$p_U$ is a polynomial with $p_U(e_1)=p(x^*)>0$. Applying the statement of the theorem to~$p_U$, we have that for every matrix  $\tilde{V}\in\reals^{n\times n}$ with $\|\tilde{V}-e_1\mathbf{1}^T\|_{F}\le \delta(p_U;e_1)$, the polynomial $p_U(\tilde{V}x)=p(U\tilde{V}x)$ has nonnegative coefficients. Hence, for every matrix $V\in\reals^{n\times n}$ with $\|V-x^*\mathbf{1}^T\|_F\le \delta(p_U;e_1)$, the matrix $U^TV$ satisfies
\[\|U^TV-e_1\mathbf{1}^T\|_{F}=\|U(U^TV-e_1\mathbf{1}^T)\|_F=\|V-x^*\mathbf{1}^T\|_F\le \delta(p_U;e_1).\]
This implies that $p(Vx)=p(U(U^TV)x)$ has nonnegative coefficients. Finally, observe that $\eta(p)$ is unitarily invariant, i.e., $\eta(p)=\eta(p_U)$ for any orthogonal matrix $U$. Consequently,
\[\begin{aligned}
\delta(p_U;e_1)
    &=\sup_{\epsilon>1/p_U(e_1)}\min\left\{\frac{(p_U(e_1)\epsilon)^{1/d}-1}{(p_U(e_1)\epsilon)^{1/d}+1},\frac{1}{(1+\eta(p_U)\epsilon)n}\right\}\\
    &=\sup_{\epsilon>1/p(x^*)}\min\left\{\frac{(p(x^*)\epsilon)^{1/d}-1}{( p(x^*)\epsilon)^{1/d}+1},\frac{1}{(1+\eta(p)\epsilon)n}\right\}\\
    &=\delta(p;x^*),
\end{aligned}\]
which completes the proof.
\end{proof}


\subsection{Proof of the Optimization-Free Disjunctive Positivstellensatz}

With the theorem from the previous subsection in place, we are now ready to prove Theorem~\ref{cor:strictly_pos_new}.
\begin{proof}[Proof of Theorem~\ref{cor:strictly_pos_new}]
We construct 
$\mcV^{1}$ using $2^n$ matrices, each of the form $\left[\pm e_1\, \ldots\, \pm e_n\right]$, so that the conic hulls of their columns correspond to the $2^n$ orthants of $\reals^n$. We describe how to further refine the partitioning of the space by focusing on the nonnegative orthant (the other orthants can be handled analogously by flipping signs). 
The columns of the matrix $\left[e_1\, \ldots\,  e_n\right]$ form the vertices of the unit simplex. 
For $m\ge 2$, we divide each edge of the unit simplex into $m$ equal segments of length $\sqrt{2}/m$ and partition the simplex into $m^{n-1}$ smaller simplices. Taking the conic hulls of these simplices partitions the nonnegative orthant into $m^{n-1}$ simplicial cones, whose extreme rays associate with the columns of $m^{n-1}$ invertible $n\times n$ matrices. Specifically, if $w_1,\ldots,w_n$ are the vertices of one of the smaller simplices, we include the invertible matrix 
$$V=\left[\frac{w_1}{\|w_1\|_2} \ \cdots\ \frac{w_n}{\|w_n\|_2}\right]$$
as a member of the simplicial disjunction $\mcV^m$. Repeating this procedure for each of the $2^n$ orthants, we obtain a total of $m^{n-1}2^n$ matrices that together define $\mcV^m$.

We now estimate how fast the maximum pairwise distance between the columns of each matrix $V$ constructed as above tends to zero as $m\to\infty$. Note that at level $m$, the vertices of the small simplices have the form $\alpha/m$, where $\alpha=(\alpha_1,\ldots,\alpha_n)$ is an $n$-tuple of nonnegative integers such that $\sum_{i=1}^n\alpha_i=m$. Let $\alpha/m$ and $\beta/m$ be two vertices of the same simplex. Since $\|\alpha/m-\beta/m\|_2=\sqrt{2}/m$, we have $\|\alpha-\beta\|_2=\sqrt{2}$. Note that their corresponding columns in $V$ are $\alpha/\|\alpha\|_2$ and $\beta/\|\beta\|_2$. The distance between these two columns can be upper bounded as
\begin{equation}\label{eq:dis_vertices}
   \begin{aligned}
\left\|\frac{\alpha
}{\|\alpha\|_2}-\frac{\beta
}{\|\beta\|_2}\right\|_2&=\left\|\frac{\alpha-\beta}{\|\alpha\|_2}+\beta\left(\frac{1
}{\|\alpha\|_2}-\frac{1
}{\|\beta\|_2}\right)\right\|_2\\
&=\left\|\frac{\alpha-\beta}{\|\alpha\|_2}+\frac{\beta}{\|\beta\|_2}\cdot\frac{\|\beta\|_2-\|\alpha\|_2
}{\|\alpha\|_2}\right\|_2\\
&\le \frac{\|\alpha-\beta\|_2}{\|\alpha\|_2}+\frac{|\|\beta\|_2-\|\alpha\|_2|}{\|\alpha\|_2}\\
&\le \frac{2\|\alpha-\beta\|_2}{\|\alpha\|_2}\le \frac{2\sqrt{2n}}{m},
\end{aligned} 
\end{equation}
where we have used the triangle inequality and the fact that 
\[\|\alpha\|_2^2\ge \frac{1}{n}\left(\sum_{i=1}^n\alpha_i\right)^2=\frac{m^2}{n}.\] 
Let $v\in\mbS^{n-1}$ denote any column of a matrix $V$ constructed as above. Then, it follows from~\eqref{eq:dis_vertices} that $\|V-v\mathbf{1}^T\|_F\le 2\sqrt{2}n/m$. 
Let $p_{\min}=\min_{x\in\mbS^{n-1}}p(x)$, and let $\eta(p)$ and $\delta(p;v)$ be as in Theorem~\ref{thm:degd_homo_new}. Define
\[\tilde{M}(p)\defeq \left\lceil2\sqrt{2}n\left(\sup_{\epsilon>1/p_{\min}}\min\left\{\frac{(p_{\min}\epsilon)^{1/d}-1}{( p_{\min}\epsilon)^{1/d}+1},\frac{1}{(1+\eta(p)\epsilon)n}\right\}\right)^{-1}\right\rceil.\]
It follows that $\|V-v\mathbf{1}^T\|_F\le 2\sqrt{2}n/\tilde{M}(p)
\le \delta(p;v)$ whenever $m\ge \tilde{M}(p)$. By Theorem~\ref{thm:degd_homo_new}, the polynomial $p(Vx)$ has nonnegative coefficients. Since the same argument applies to every matrix $V$ constructed as above, we conclude that $p\in \DiSOSL_{n,d}(\mcV^{m})$ whenever $m\ge \tilde{M}(p)$. 
\end{proof}


Analogous to Algorithm~\ref{alg:uniform_grid}, which is based on Theorem \ref{cor:strictly_pos}, the optimization-free disjunctive Positivstellensatz can be used to construct 
an optimization-free algorithm for certifying polynomial nonnegativity. In practice, instead of uniformly refining each simplicial cone as in the proof, it is more efficient to adaptively refine these cones within a spatial branch-and-bound framework. Even with such a modification, we expect the optimization-free approach to require a significantly larger number of regions to obtain a disjunctive certificate of nonnegativity compared to the disjunctive sum of squares approach based on semidefinite programming (i.e., the approach of Section~\ref{sec:disos_psatz}). Studying these tradeoffs computationally is a direction for future research.

\section{Extensions to Constrained Optimization Problems}
\label{sec:constrained_opt}
In this section, we extend the disjunctive sum of squares method from certifying global nonnegativity of polynomials to certifying nonnegativity over a closed basic semialgebraic set, i.e., a subset of the Euclidean space defined by finitely many polynomial inequalities. 
We present two approaches. The first, discussed in Section~\ref{sec:hierarchy}, builds on a polynomial-time reduction from~\cite{aaa2019} that converts constrained polynomial optimization problems (POPs) into unconstrained ones, to which our disjunctive sum of squares method then applies. The second, given in Section~\ref{sec:copositive}, extends Definition~\ref{def:disjunction} so that the subregions defined by the algebraic disjunction cover the feasible set of interest rather than all of $\reals^n$. We illustrate this approach for the problem of testing matrix copositivity, which reduces to minimizing a quadratic form over the nonnegative orthant.

\subsection{Lower Bounds on the Optimal Values of Polynomial Optimization Problems}\label{sec:hierarchy}
Consider a POP of the form
\begin{equation}\label{eq:general_pop}
	\begin{array}{rl}
    \min\limits_{x\in{\scriptsize\reals^n}} & p(x)\\
		\st & g_j(x)\ge 0, \quad j=1,\ldots,J,
	\end{array}    
\end{equation}
with a compact feasible set. Throughout this section, we let $p^*$ denote the optimal value of problem~\eqref{eq:general_pop}.  
We show that our disjunctive Positivstellens\"atze can be combined with the reduction from~\cite[Theorem 2.1]{aaa2019} to produce a converging hierarchy of lower bounds on $p^*$. 

We first describe the reduction from~\cite{aaa2019}. 
Let $d$ be the smallest even integer larger or equal to the maximum degrees of the polynomials $p$ and $g_1,\ldots,g_J$. Following~\cite[Theorem 2.1]{aaa2019}, one can explicitly construct a form $f_{\gamma}(z)$ in $N\defeq n+J+3$ variables $z$ and of degree $2d$ using a given scalar $\gamma$, the coefficients of $p, g_1,\ldots,g_J$, and an upper bound on the radius of a ball containing the feasible set of problem~\eqref{eq:general_pop}. This construction has the property that $\gamma$ is a strict lower bound on $p^*$ if and only if $f_{\gamma}(z)$ is positive definite. 

Then,~\cite[Theorem 2.4]{aaa2019} shows how to use this reduction to solve problem~\eqref{eq:general_pop}. 
Recall that $H_{n,d}$ denotes the set of homogeneous polynomials in~$n$ variables and degree~$d$, and $\HNN_{n,d}$ denotes the set of nonnegative forms in $H_{n,d}$. 
Consider a sequence of subsets of $H_{n,d}$, denoted by $K_{n,d}^i$ and indexed by $i$, that satisfies the following properties:
\begin{enumerate}[label=(\arabic*)]
    \item $K_{n,d}^i\subseteq \HNN_{n,d}$ for all $i$, and there exists a positive definite form $s_{n,d}\in K^0_{n,d}$,
    \item if $p\in H_{n,d}$ is positive definite, then there exists an integer $i$ such that $p\in K_{n,d}^i$,
    \item $K_{n,d}^i\subseteq K_{n,d}^{i+1}$ for all $i$,
    \item if $p\in K_{n,d}^i$, then $p + \epsilon s_{n,d}\in K_{n,d}^i$ for all $\epsilon\in[0,1]$.
\end{enumerate}

By~\cite[Theorem 2.4]{aaa2019}, 
for any such sequence $K_{n,d}^i$ satisfying properties~(1)--(4), the sequence $l_i$ defined by
\begin{equation}\label{eq:general_hierarchy}
     	\begin{array}{rl}
		l_i \defeq \max\limits_{\gamma\in\reals} & \gamma\\
		\st & f_{\gamma}(z)-\dfrac{1}{i}s_{N,2d}(z)\in K_{N,2d}^i,
	\end{array}  
\end{equation}
is non-decreasing, satisfies $l_i\le p^*$ for all~$i$, and $\lim_{i\to\infty} l_i=p^*$. 
As shown in~\cite{aaa2019}, 
Artin, Reznick, and Poly\'a's Positivstellens\"atze all lead to sequences $K_{n,d}^i$ satisfying properties~(1)--(4).
In Theorem~\ref{thm:hier_pop} below, we show that our disjunctive sum of squares method fits into this framework as well. 
Moreover, in contrast to the above Positivstellens\"atze, our method has the 
advantage that the degrees of the polynomials required to certify membership in $K_{n,d}^i$ remain fixed (and low) throughout the hierarchy.

We recall that $\HSOS_{n,d}$ denotes the set of sos forms in~$n$ variables and degree~$d$.



\begin{theorem}\label{thm:hier_pop}
Fix a dimension $n$ and an even degree $d$. Let $\mcD^m$ be any family of algebraic disjunctions, indexed by $m$, that satisfies the guarantees of Theorem~\ref{cor:strictly_pos}. 
Let $\widehat{K}_{n,d}^0\defeq\HSOS_{n,d}$. For each integer $i\ge 1$, define
    \begin{equation}\label{eq:hier_pop}
      \widehat{K}_{n,d}^i\defeq  \bigcup_{m=1}^i\DiSOS_{n,d}^d\left(\mcD^m\right) .
    \end{equation}
Then, the sequence $\widehat{K}_{n,d}^i\cap H_{n,d}$ satisfies properties~(1)--(4) stated above.
\end{theorem}

\begin{proof}
We verify properties~(1)--(4). 
\begin{enumerate}[label=(\arabic*)]
        \item When $i=0$, $\widehat{K}_{n,d}^0=\HSOS_{n,d}\
        \subseteq \HNN_{n,d}$. 
        Let $i\ge 1$.         
        Since $\SOS_{n,d}\subseteq\DiSOS^d_{n,d}(\mcD^m)\subseteq\NN_{n,d}$ for all $m$, we have
        $\HSOS_{n,d}\subseteq\DiSOS^d_{n,d}(\mcD^m)\cap H_{n,d}\subseteq\HNN_{n,d}$ for all $m$. Taking unions gives 
        $\HSOS_{n,d}\subseteq \widehat{K}_{n,d}^i\cap H_{n,d}\subseteq \HNN_{n,d}$ for all $i\ge 1$. Moreover, the positive definite form~$\|x\|_2^d\in\HSOS_{n,d}=\widehat{K}_{n,d}^0$.
        \item By Theorem~\ref{cor:strictly_pos}, for any positive definite form~$p$, there exists an integer $\bar{m}$ such that $p\in \DiSOS^d_{n,d}(\mcD^{\bar{m}})$, and therefore $p\in\widehat{K}_{n,d}^{\bar{m}}$.
        \item The sequence $\widehat{K}_{n,d}^i$ is nested by construction.        
        \item When $i=0$, the property holds trivially. 
        When $i\ge 1$, suppose $p\in\widehat{K}_{n,d}^i$. Then, there exists $m\in\{1,\ldots,i\}$ such that $p\in\SOS_{n,d}^d(\mcD^m)$. 
        Recall from Theorem~\ref{cor:strictly_pos} that each subset in $\mcD^m$ contains exactly one form, i.e., we can write
        \[\mcD^m=\big\{\{q_{j}(x)\},\quad j=1,\ldots,r_m\big\}.\]
 For each $j\in\{1,\ldots,r_m\}$, 
        there exist sos polynomials $s_{j}$ and $\sigma_{j}$ such that
        \[p(x)=s_{j}(x)+\sigma_{j}(x)q_{j}(x),\]
        with $\deg s_{j}\le d$ and $\deg (\sigma_{j}q_{j})\le d$. 
        Consequently, 
        \[p(x) + \epsilon \|x\|_2^d=(s_{j}(x)+\epsilon \|x\|_2^d)+ \sigma_{j}(x)q_{j}(x).\]
        Since $\|x\|_2^d$ is sos, the polynomial $s_{j}(x)+\epsilon \|x\|_2^d$ remains sos for all $\epsilon\in[0,1]$ and $j=1,\ldots,r_m$. Therefore, 
        $p(x) + \epsilon \|x\|_2^d\in \SOS_{n,d}^d(\mcD^m)\subseteq\widehat{K}_{n,d}^i$  for all $\epsilon\in[0,1]$.
    \end{enumerate}
\vspace{-\baselineskip}
\end{proof}
Since the sequence $\widehat{K}_{n,d}^i$ satisfies properties~(1)--(4), replacing\footnote{
To compute the optimal value $l_i$ of problem~\eqref{eq:general_hierarchy} with $K_{N,2d}^i$ replaced by $\widehat{K}_{N,2d}^i$, one can take the maximum between $l_{i-1}$ and the optimal value of problem~\eqref{eq:general_hierarchy} with $K_{N,2d}^i$ replaced by $\DiSOS^{2d}_{N,2d}(\mcD^i)$. As discussed earlier, a membership constraint to $\DiSOS^{2d}_{N,2d}(\mcD^i)$ amounts to solving a semidefinite program.
} $K_{N,2d}^i$ with $\widehat{K}_{N,2d}^i$ in~\eqref{eq:general_hierarchy} and applying~\cite[Theorem 2.4]{aaa2019} yields a non-decreasing sequence $l_i$ that satisfies 
$l_i\le p^*$ for all $i$ and $\lim_{i\to\infty} l_i=p^*$. 
Furthermore, one can show that 
 the theorem remains valid if the cones $\DiSOS_{n,d}^d(\mcD^m)$ in~\eqref{eq:hier_pop} are replaced by the cones $\DiSOSL_{n,d}(\mcV^m)$ introduced in Theorem~\ref{cor:strictly_pos_new}, yielding an optimization-free hierarchy for general POPs with a compact feasible set.

\subsection{An Application to Copositive Programming}\label{sec:copositive}
In this section, we give an example of how the disjunctive sum of squares approach can be adapted to certify nonnegativity over a subset of the Euclidean space. We take this subset to be the nonnegative orthant, which arises naturally in copositive programming.

A symmetric matrix $Q \in \reals^{n \times n}$ is said to be \emph{copositive} if $x^T Q x \geq 0$ for all $x \geq 0$. Let $\mcC_n$ denote the cone of $n\times n$ copositive matrices. A copositive program is a linear optimization problem over the intersection of $\mcC_n$ with an affine subspace. 
In recent years, copositive programming has found many applications in both discrete and continuous optimization; see, e.g.,~\cite{burer2012copositive,dur2010copositive}. 
Since checking copositivity of a given matrix is NP-hard~\cite{murty1987some}, tractable inner approximations to $\mcC_n$ are of interest in most applications of copositive programming. 

A well-known SDP-based sufficient condition for copositivity is the so-called {``$P+N$'' criterion}. For a symmetric matrix $A \in \reals^{n \times n}$, let us use the standard notation $A \succeq 0$ (resp. $A \geq 0$) to denote that $A$ is positive semidefinite (resp. elementwise nonnegative). Observe that if a matrix $Q\in\reals^{n\times n}$ can be written as $Q = P + N$, where $P \succeq 0$ and $N \geq 0$, then $Q$ is copositive. 
This condition is known to also be necessary for copositivity if $n\le 4$~\cite{Diananda1962}. For $n\ge 5$, however, there are copositive matrices that do not admit a $P+N$ decomposition. A classical example is the Horn matrix~\cite{Diananda1962}:
\begin{equation}\label{eq:horn}
  H=\begin{bmatrix}
    \phantom{-}1 & -1 & \phantom{-}1 & \phantom{-}1 & -1 \\
    -1 & \phantom{-}1 & -1 & \phantom{-}1 & \phantom{-}1 \\
    \phantom{-}1 & -1 & \phantom{-}1 & -1 & \phantom{-}1 \\
    \phantom{-}1 & \phantom{-}1 & -1 & \phantom{-}1 & -1 \\
    -1 & \phantom{-}1 & \phantom{-}1 & -1 & \phantom{-}1 \\
  \end{bmatrix}.
\end{equation}


To bridge the gap between the $P+N$ criterion and copositivity, a complete SDP-based hierarchy of sufficient conditions
for copositivity was introduced in~\cite{Parrilo2000}.
Let $x.^2\defeq(x_1^2,\ldots,x_n^2)$. At level $m$ of this hierarchy, one checks if the polynomial $(x.^2)^TQ(x.^2)\cdot(x_1^2 + \cdots + x_n^2)^m$ is sos, and if so, copositivity is certified. The zeroth level of this hierarchy recovers exactly the $P+N$ criterion described above~\cite{Parrilo2000}.
Moreover, this hierarchy is complete: for any strictly copositive matrix $Q$ (i.e., a matrix $Q$ satisfying $x^TQx> 0$ for all $x\ge 0$ with $x\neq 0$), there exists a level $m$ at which copositivity is certified~\cite{Parrilo2000, reznick1995uniform}.
However, the size of the semidefinite constraint at level $m$ of this hierarchy grows as $\binom{n+m+2}{n} \times \binom{n+m+2}{n}$. While one can reduce this size using symmetry reduction techniques~\cite[Section 8.1]{Gatermann2004}, the higher levels of this hierarchy often remain too expensive to solve. 

Here, we utilize the disjunctive sos approach to propose an alternative SDP-based complete hierarchy for certifying copositivity of an $n\times n$ matrix, where the size of the semidefinite constraint remains $n\times n$ throughout the hierarchy. To this end, we extend the notion of a simplicial disjunction in Definition~\ref{def:disos_new} and require the subregions to cover only the nonnegative orthant rather than all of $\reals^n$.

\begin{definition}\label{def:sd_simplex}
Let $r$ be a positive integer. We say that a collection of $n\times n$ invertible matrices $\mcV=\{V_k\}_{k=1,\ldots,r}$ forms a \emph{simplicial disjunction} of a region $\Omega\subseteq\reals^n$ if 
\[\Omega\subseteq \bigcup\limits_{k=1}^r\Omega_k,\quad \text{where }\quad\Omega_k=\cone\{V_k\}\defeq\cone\{v_{k,1}, \ldots, v_{k,n}\},\]
and $v_{k,j}$ is the $j$-th column of $V_k$.
\end{definition}

\begin{definition}\label{def:copos_disos}
A symmetric matrix $Q\in\reals^{n\times n}$ is \emph{disjunctive $P+N$ with respect to a simplicial disjunction} $\mcV=\{V_k\}_{k=1,\ldots,r}$ of the nonnegative orthant if there exist matrices $P_k$ and $N_k$, for $k=1,\ldots,r$, such that
\[V_k^TQV_k=P_k+N_k,\;P_k\succeq 0,\; N_k\ge 0,\quad k=1,\ldots,r.\]
We denote the set of such matrices by $\mcPN_{n}(\mcV)$. 
\end{definition}


Observe that for any simplicial disjunction $\mcV$ of the nonnegative orthant, the set $\mcPN_n(\mcV)$ is a semidefinite representable convex cone. Moreover, it is straightforward to see that $\mcPN_n(\mcV)\subseteq\mcC_n$. Indeed, 
if $Q\in\mcPN_n(\mcV)$,  
the polynomial $y^TV_k^TQV_ky$ is nonnegative for all $y\ge 0$ and $k=1,\ldots,r$. This implies that $x^TQx$ is nonnegative on each subregion $\Omega_k=\cone\{V_k\}$ as any point $x\in\cone\{V_k\}$ can be written as $x=V_ky$ with $y\ge 0$. Since these subregions cover the nonnegative orthant, we have $x^TQx\ge 0$ 
for all $x\ge 0$, and hence $Q$ is copositive.

Let us revisit the Horn matrix $H$ in~\eqref{eq:horn} in view of Definition~\ref{def:copos_disos}. 
Consider the simplicial disjunction $\mcV=\{V_1,V_2\}$ of the nonnegative orthant, where
\[V_1=\begin{bmatrix}
     1   & 0&    0&     0&   0\\
    0 &    1&     0&    0&     0\\
   0&   0&     1&   0&     0\\
    0 &   0&    0&     1&    0.5\\
    0 &    0&     0&    0&     0.5\\
\end{bmatrix}, \quad 
V_2=\begin{bmatrix}     
     1   & 0&    0&     0&   0\\
    0 &    1&     0&    0&     0\\
   0&   0&     1&   0&     0\\
    0 &   0&    0&     0.5&    0\\
    0 &    0&     0&    0.5&     1\\ \end{bmatrix}.\]
One can readily verify that
\[
V_1^THV_1=\begin{bmatrix}
    \phantom{-}1 & -1            & \phantom{-}1  & \phantom{-}1  & \phantom{-}0 \\
    -1            & \phantom{-}1  & -1            & \phantom{-}1  & \phantom{-}1 \\
    \phantom{-}1 & -1            & \phantom{-}1  & -1            & \phantom{-}0 \\
    \phantom{-}1 & \phantom{-}1  & -1            & \phantom{-}1  & \phantom{-}0 \\
    \phantom{-}0 & \phantom{-}1  & \phantom{-}0  & \phantom{-}0  & \phantom{-}0 \\
\end{bmatrix}=\begin{bmatrix}
    \phantom{-}1 & -1            & \phantom{-}1  & -1            & \phantom{-}0 \\
    -1            & \phantom{-}1  & -1            & \phantom{-}1  & \phantom{-}0 \\
    \phantom{-}1 & -1            & \phantom{-}1  & -1            & \phantom{-}0 \\
    -1            & \phantom{-}1  & -1            & \phantom{-}1  & \phantom{-}0 \\
    \phantom{-}0 & \phantom{-}0  & \phantom{-}0  & \phantom{-}0  & \phantom{-}0 \\
\end{bmatrix}+\begin{bmatrix}
     0 & 0 & 0 & 2 & 0 \\
     0 & 0 & 0 & 0 & 1 \\
     0 & 0 & 0 & 0 & 0 \\
     2 & 0 & 0 & 0 & 0 \\
     0 & 1 & 0 & 0 & 0 \\
\end{bmatrix},
\]
\[
V_2^THV_2=\begin{bmatrix}
    \phantom{-}1 & -1            & \phantom{-}1  & \phantom{-}0  & -1            \\
    -1            & \phantom{-}1  & -1            & \phantom{-}1  & \phantom{-}1 \\
    \phantom{-}1 & -1            & \phantom{-}1  & \phantom{-}0  & \phantom{-}1 \\
    \phantom{-}0 & \phantom{-}1  & \phantom{-}0  & \phantom{-}0  & \phantom{-}0 \\
    -1            & \phantom{-}1  & \phantom{-}1  & \phantom{-}0  & \phantom{-}1 \\
\end{bmatrix}=\begin{bmatrix}
    \phantom{-}1 & -1            & \phantom{-}1  & \phantom{-}0  & -1            \\
    -1            & \phantom{-}1  & -1            & \phantom{-}0  & \phantom{-}1 \\
    \phantom{-}1 & -1            & \phantom{-}1  & \phantom{-}0  & -1            \\
    \phantom{-}0 & \phantom{-}0  & \phantom{-}0  & \phantom{-}0  & \phantom{-}0 \\
    -1            & \phantom{-}1  & -1            & \phantom{-}0  & \phantom{-}1 \\
\end{bmatrix}+\begin{bmatrix}
     0 & 0 & 0 & 0 & 0 \\
     0 & 0 & 0 & 1 & 0 \\
     0 & 0 & 0 & 0 & 2 \\
     0 & 1 & 0 & 0 & 0 \\
     0 & 0 & 2 & 0 & 0 \\
\end{bmatrix}.
\]
Let $e_1,\ldots,e_5$ be the standard basis vectors in $\reals^5$. The first identity proves nonnegativity of $x^THx$ over $\cone\{V_1\}=\cone\{e_1,e_2,e_3,e_4,(e_4+e_5)/2\}$. Similarly, the second identity proves nonnegativity of $x^THx$ over $\cone\{V_2\}=\cone\{e_1,e_2,e_3,(e_4+e_5)/2,e_5\}$. 
Together, these two identities certify that the Horn matrix $H\in\mcPN_5(\mcV)$ and is hence copositive. 

We show in Theorem~\ref{thm:psatz_copositive_full} below that one can a priori fix a family of simplicial disjunctions of the nonnegative orthant such that any strictly copositive matrix is disjunctive $P+N$ with respect to some simplicial disjunction in the family. 
In fact, the proof shows that the disjunctive proofs do not necessarily need to involve positive semidefinite matrices; that is, one can set $P_k = 0$ in Definition~\ref{def:copos_disos}. This result follows as a simple corollary of what we have proven in the previous section.


In the proof of the following theorem, we use the standard notation~$\textcolor{black}{\Delta^{n-1}}$ to denote the unit simplex in~$\reals^n$, i.e., $\textcolor{black}{\Delta^{n-1}}=\{x\in\reals^n\mid \sum_{i=1}^n x_i=1,\; x\ge 0\}$, and we let $\conv\{V\}$ denote the convex hull of the columns of a matrix
$V\in\reals^{n\times n}$.

\begin{theorem}[Disjunctive Positivstellensatz for Matrix Copositivity]
\label{thm:psatz_copositive_full}
For any fixed dimension $n$, there exists an explicit family of simplicial disjunctions
$\mcV^m$ of the nonnegative orthant, indexed by $m$, such that for any strictly 
copositive matrix $Q \in \reals^{n \times n}$, we have $Q \in \mcPN_{n}(\mcV^m)$ 
for some $m$.
\end{theorem}
\begin{proof}
The construction of the family of simplicial disjunctions $\mcV^m$ is similar to the construction for Theorem~\ref{cor:strictly_pos_new}. We let $\mcV^1=\{I_n\}$. For $m\ge 2$, we divide each edge of the unit simplex $\Delta^{n-1}$ into $m$ equal segments of length $\sqrt{2}/m$ and partition the simplex into $m^{n-1}$ smaller simplices. For each smaller simplex, we form a matrix whose columns are the vertices of this simplex and include this matrix as a member of $\mcV^m$. 

Let $q(x)\defeq x^TQx$. 
Since $q$ is homogeneous and $Q$ is strictly copositive, we have $q(x^*/\|x^*\|_2) > 0$ for all $x^* \in \textcolor{black}{\Delta^{n-1}}$. Letting $\delta(\cdot;\cdot)$ be defined as in Theorem~\ref{thm:degd_homo_new}, it follows that $\delta(q;\, x^*/\|x^*\|_2) > 0$ for all
$x^* \in \textcolor{black}{\Delta^{n-1}}$. It is straightforward to check that $\delta_{\min}\defeq \inf_{x^* \in \textcolor{black}{\Delta^{n-1}}}
\|x^*\|_2\,\delta(q; x^*/\|x^*\|_2) > 0$. 
For any integer $m \geq \lceil\sqrt{2n}/\delta_{\min}\rceil$ and any matrix $V \in \mcV^m$, $\conv\{V\}$ is a polytope with edge length $\sqrt{2}/m$, so for any column $v \in \textcolor{black}{\Delta^{n-1}}$ of $V$, we have
\[
\|V - v\mathbf{1}^T\|_F \leq \sqrt{n} \cdot \frac{\sqrt{2}}{m} \leq \delta_{\min}
\leq \|v\|_2\,\delta(q; v/\|v\|_2).
\]
Thus, $\left\|V/\|v\|_2 - (v/\|v\|_2)\mathbf{1}^T\right\|_F \leq \delta(q; v/\|v\|_2)$, and by Theorem~\ref{thm:degd_homo_new}, the polynomial $q(V x/\|v\|_2) = x^TV^TQVx/\|v\|_2^2$ 
has nonnegative coefficients, hence $V^TQV \geq 0$.
Since the same argument holds for each matrix in $\mcV^m$, we conclude that 
$Q \in \mcPN_n(\mcV^m)$.
\end{proof}

\begin{remark}
We make two observations regarding Theorem~\ref{thm:psatz_copositive_full} and Definition~\ref{def:copos_disos}:
\begin{enumerate}
\item [(i)]
By the construction in the proof of Theorem~\ref{thm:psatz_copositive_full}, the first simplicial disjunction $\mcV^1$ contains only the identity matrix, 
and hence checking membership in $\mcPN_n(\mcV^1)$ reduces exactly to the $P+N$ test.
Moreover, the tests for copositivity in the higher levels of the hierarchy all dominate the $P+N$ test. Indeed, consider an arbitrary integer $m\ge 1$. If $Q = P + N$ with $P \succeq 0$ and $N \geq 0$, then for each matrix $V$ in $\mcV^m$, $V^T Q V = V^T P V + V^T N V$, where $V^T P V \succeq 0$, and $V^T N V \geq 0$ as the entries of $V$ are nonnegative by construction. Hence, $Q \in \mcPN_n(\mcV^m)$.

\item [(ii)]
Setting $P_k = 0$ in Definition~\ref{def:copos_disos} reduces the condition 
$V_k^T Q V_k= P_k + N_k$ to $V_k^T Q V_k~\geq~0$, 
which coincides with the sufficient condition for copositivity in~\cite[Theorem 1]{BUNDFUSS2008}.
However, by doing so, no level of the resulting hierarchy would dominate the $P+N$ test. See also Section~\ref{sec:num_cop} for some related comparisons in numerical experiments.
\end{enumerate}
\end{remark}

\section{A Spatial Branch-and-Bound Framework}\label{sec:bnb}

While the Positivstellens\"atze in previous sections show that one can always find low-degree disjunctive sos proofs of nonnegativity of a polynomial by uniformly dividing the space, in practice it is more efficient to perform this division \emph{adaptively}. In this section, we describe how this can be done by integrating our framework into a spatial branch-and-bound (SBB) algorithm.
We present SBB algorithms to certify polynomial nonnegativity and matrix copositivity in Section~\ref{sec:adaptive_grid} and Section~\ref{sec:bnb_simplex}, respectively, and evaluate their numerical performance in Section~\ref{sec:num}. 

\subsection{A SBB Framework for Polynomial Minimization}
\label{sec:adaptive_grid}
In this section, we apply a branch-and-bound algorithm to approximate the minimum value $p_{\min}$ of a form $p\in H_{n,d}$ over the unit sphere $\mbS^{n-1}$ by adaptively dividing the sphere into subregions using simplicial disjunctions.
{We use simplicial disjunctions because it is straightforward to partition a simplicial cone into two smaller ones.
}
We consider two choices for the initial simplicial disjunction:
\begin{enumerate}[label=(\arabic*)]
    \item $2^{n-1}$ matrices, each of the form\footnote{Since even-degree homogeneous polynomials are symmetric with respect to the origin, it suffices to certify nonnegativity on a hemisphere.} $\begin{bmatrix}\pm e_1 & \ldots & \pm e_{n-1} & e_n\end{bmatrix}$.
    \item $n+1$ matrices, constructed as follows: Consider $n+1$ vertices of a regular $n$-simplex
\[c_i=\sqrt{1+n^{-1}}e_i-n^{-3/2}(\sqrt{n+1}-1)\mathbf{1},\quad i=1,\ldots,n,\quad c_{n+1}=-n^{-1/2}\mathbf{1}.\]
Each matrix is formed by taking $n$ of the $n+1$ vectors above as its columns.
\end{enumerate}

\begin{figure}[!tb]
    \centering
    \includegraphics[width=0.85\textwidth]{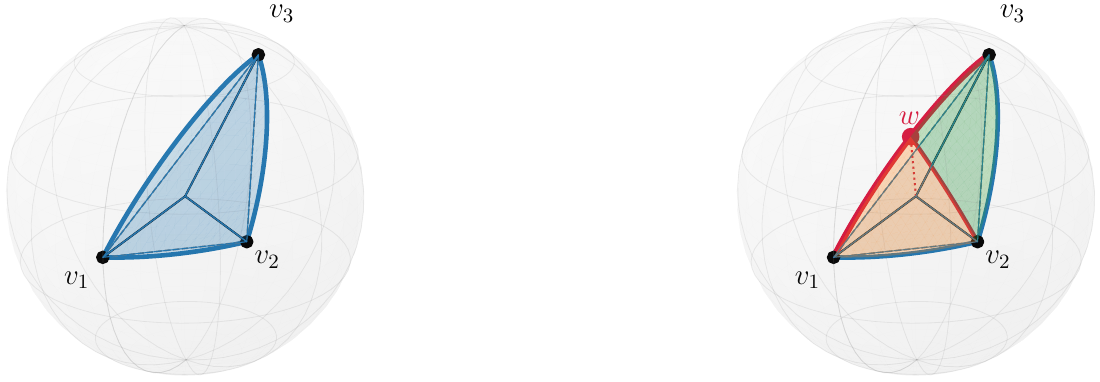}
    \caption{Branching step for $\cone\{V\}=\cone\{v_1,v_2,v_3\}$. Left: $\mcQ=\cone\{V\}\cap\mbS^{2}$ is the spherical triangle bounded by geodesic arcs (solid); planar chords of~$\cone\{V\}$ shown dashed. Right: the longest arc~$(v_1,v_3)$ is bisected at its geodesic midpoint~$w$, splitting~$\mcQ$ into~$\mcQ^+$ and~$\mcQ^-$.}
    \label{fig:branching}
\end{figure}

\paragraph{Bounding.}
Each subregion of the sphere takes the form $\mcQ = \cone\{V\} \cap \mbS^{n-1}$ for some invertible
matrix $V\in \reals^{n \times n}$. We compute {a} lower bound on $p$ over $\mcQ$ by solving
\begin{equation}\label{eq:poly_bnb}
	\phi(V) \defeq \begin{array}[t]{cl}
		\max\limits_{\gamma\in\reals} & \gamma\\
		\st & p(V(x.^2))-\gamma\|V(x.^2)\|_2^d\in\HSOS_{n,2d}.
	\end{array}
\end{equation} 
Note that any feasible $\gamma$ in problem~\eqref{eq:poly_bnb} is a valid lower bound on~$p$ over~$\mcQ$. 
{Indeed, any point $z\in\cone\{V\}$ can be written as $z=V(x.^2)$ for some $x\in\reals^n$, so the constraint in~\eqref{eq:poly_bnb} implies $p(z)-\gamma\|z\|_2^d\ge 0$ for all $z\in\cone\{V\}$. In particular, $p(z)-\gamma\ge 0$ for all $z\in\cone\{V\} \cap \mbS^{n-1}$.} 

{To obtain an upper bound on~$p$ over~$\mcQ$, we run $K$ projected gradient descent steps (where $K$ is chosen by the user): starting from a column of~$V$, we iterate}
\begin{equation}\label{eq:pgd_poly}
   y^k=\Pi_{\cone\{V\}}(x^k - \beta \nabla p(x^k)),\quad x^{k+1} = \frac{y^k}
    {\|y^k\|_2},\quad k=0,\ldots,K-1,
\end{equation}
where $\Pi_{\cone\{V\}}$ denotes the projection onto $\cone\{V\}$ and $\beta>0$ is a chosen stepsize.


\paragraph{Branching.}
At each iteration, we select a subregion
\begin{equation*}
\mcQ=\cone\{V\}\cap\mbS^{n-1}
\end{equation*}
 with the smallest lower bound, where $V = [v_1 \,\ldots\, v_n] \in\reals^{n\times n}$ is some invertible matrix. Letting $(i,j)\in \operatornamewithlimits{argmax}_{1 \le k, l \le n} \|v_k-v_l\|_2$, we 
set~$w= (v_{i}+v_{j})/\|v_{i}+v_{j}\|_2$ and form two further subregions
\[
    \mcQ^+=\cone\{V^+\}\cap \mbS^{n-1},\quad \mcQ^-=\cone\{V^-\}\cap \mbS^{n-1},\]
where $V^+=[v_{1}\, \ldots\, v_{i-1}\, v_{i+1}\, \ldots\, v_{n}\, w]$ and $V^-=[v_{1}\, \ldots\, v_{j-1}\, v_{j+1}\, \ldots\, v_{n}\, w]$.
{
Figure~\ref{fig:branching} illustrates an example of this construction when $n=3$.}

\textcolor{black}{The 
overall branch-and-bound scheme is presented in Algorithm~\ref{alg:bnb}. 
One can utilize Theorem~\ref{thm:degd_homo_new} together with convergence results for branch-and-bound methods (see, e.g.~\cite[Theorem~5.26]{locatelli2013global} or~\cite[Chapter~IV]{horst_global_1996}) to show that for any prescribed tolerance $\epsilon_{\text{tol}}>0$, Algorithm~\ref{alg:bnb} is guaranteed to terminate.}

\begin{algorithm}[H]
\caption{Spatial Branch-and-Bound Algorithm for Minimizing a Form}
\label{alg:bnb}
\begin{algorithmic}
\Require A polynomial $p\in H_{n,d}$, a tolerance $\epsilon_{\text{tol}}>0$.
\Ensure A lower bound $L$ on the minimum value of $p$ over $\mbS^{n-1}$ and a point $\bar{x}\in\mbS^{n-1}$ such that 
$p(\bar{x})-L\le \epsilon_{\text{tol}}(1+|L|+|p(\bar{x})|)$.
\Initialization Generate an initial simplicial disjunction $\{V_i\}_{i=1,\ldots,r_0}$.
\State $\mcT\gets\{V_i\mid i=1,\ldots,r_0\}$. 
\Comment{Create a branch-and-bound tree}
\State $l_i\gets \phi(V_i)$, $\quad i=1,\ldots,r_0$. \Comment{Local lower bound computed by solving~\eqref{eq:poly_bnb}}

\State $u_i\gets\min_{1\le j\le n}\ p(V_ie_j)$, $\quad i=1,\ldots,r_0$. \Comment{Local upper bound}
\State $L\gets\min_{1\le i\le r_0}\{l_i\},\quad U\gets\min_{1\le i\le r_0}\{u_i\}$. 
\Comment{Global lower and upper bounds}
\While{$U-L> \epsilon_{\text{tol}}(1+|L|+|U|)$}
\State $V\gets \argmin_{\hat{V}\in \mcT} \phi(\hat{V}),\quad \mcT\gets \mcT\setminus\{V\}$. 
\State $v_1, \ldots, v_n \gets$ columns of $V$.
\State $(i, j) \gets \operatornamewithlimits{argmax}_{1 \le k, l \le n} \|v_k-v_l\|_2,\quad w \gets (v_i + v_j)/\|v_i + v_j\|_2$.
\State $V^+ \gets [v_1\, \ldots\, v_{i-1}\, v_{i+1}\, \ldots\, v_n\, w],\quad V^- \gets [v_1\, \ldots\, v_{j-1}\, v_{j+1}\, \ldots\, v_n\, w]$. \Comment{Branch}
\State $\mcT \gets \mcT \cup \{V^+,V^-\}$.
\State $L\gets \min_{\hat{V}\in \mcT} \phi(\hat{V})$.  \Comment{Update global lower bound}
\For{$\hat{V}\in\{V^+,V^-\}$}\Comment{Update global upper bound}
\State $x^K \gets$ $K$-th iterate of~\eqref{eq:pgd_poly} with $x^0=w$.
\State $u\gets p(x^K)$.
\If{$u<U$} $U\gets u,\quad \bar{x}\gets x^K$.\EndIf
\EndFor
\EndWhile
\State \Return $L,\bar{x}$
\end{algorithmic}
\end{algorithm}



\subsection{A SBB Framework for Matrix Copositivity}\label{sec:bnb_simplex}
In this section, we focus on the problem of testing copositivity of a matrix $Q\in\reals^{n\times n}$, which reduces to computing the minimum value $p_{\min}$ of a quadratic form $p(x)=x^TQx$ over the unit simplex $\textcolor{black}{\Delta^{n-1}}$.
We develop a branch-and-bound algorithm analogous to Algorithm~\ref{alg:bnb}, tailored to the unit simplex. This algorithm can be used more generally to solve copositive programs. We initialize the algorithm with a single subregion, the unit simplex $\textcolor{black}{\Delta^{n-1}}=\conv\{I_n\}$, corresponding to the simplicial disjunction $\{I_n\}$.


\begin{figure}[!ht]
    \centering
    \includegraphics[width=0.85\textwidth]{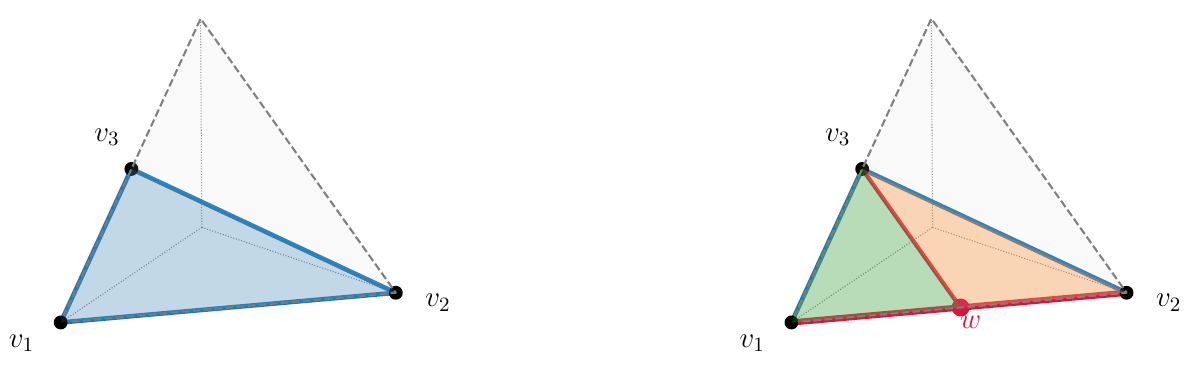}
    \caption{Branching step for $\conv\{V\}=\conv\{v_1,v_2,v_3\}$. Left: $\mcQ=\conv\{V\}$ is the blue triangle; the unit simplex~$\textcolor{black}{\Delta^2}$ shown dashed. Right: the longest edge~$(v_1,v_2)$ is bisected at its midpoint~$w=(v_1+v_2)/2$, splitting~$\mcQ$ into~$\mcQ^+$ and~$\mcQ^-$.}
    \label{fig:branching_simplex}
\end{figure}

\paragraph{Bounding.}
Each subregion of the unit simplex takes the form $\mcQ = \conv\{V\}$ for some invertible
matrix $V~=~[v_1 \,\ldots\, v_n]\in\reals^{n \times n}$. Note that the minimum of~$p$ over~$\mcQ$ equals the optimal value of
\begin{equation*}
    \begin{array}{cl}
        \max\limits_{t\in\reals} & t\\
        \st & V^T(Q-tJ_n)V\in\mcC_n,
    \end{array}
\end{equation*}
where $J_n$ is the $n\times n$ all-one matrix and $\mcC_n$ is the cone of $n\times n$ copositive matrices. Using the inner approximations of $\mcC_n$ from Section~\ref{sec:copositive}, we compute a lower bound on~$p$ over $\mcQ$ by solving the following semidefinite program:
\begin{equation}\label{eq:DiSOS_lb_mat}
    \psi(V)\defeq \begin{array}[t]{cl}
        \max\limits_{t\in\reals,N\in\reals^{n\times n}} & t\\
        \st & V^T(Q-tJ_n)V-N\succeq 0,\\
            & N\ge 0.
    \end{array}
\end{equation}
\textcolor{black}{To obtain an upper bound on~$p$ over $\mcQ$, we run $K$ projected gradient descent steps starting from a column of~$V$:}
\begin{equation}\label{eq:pgd_copos}
    x^{k+1} = \Pi_{\conv\{V\}}(x^k - 2\beta Qx^k),\quad k=0,\ldots,K-1.
\end{equation}
Here, $\Pi_{\conv\{V\}}$ denotes the projection onto $\conv\{V\}$ 
and $\beta>0$ is the stepsize.
\paragraph{Branching.}
At each iteration, we select a subregion with the smallest lower bound and split it by bisecting its longest edge. More specifically, if 
\begin{equation*}
\mcQ=\conv\{V\}
\end{equation*}
denotes the selected subregion, where $V = [v_1 \,\ldots\, v_n] \in\reals^{n\times n}$ is some invertible matrix, we let $(i,j)\in \operatornamewithlimits{argmax}_{1 \le k, l \le n} \|v_k-v_l\|_2$, and set~$w= (v_{i}+v_{j})/2$, forming two further subregions
\[
    \mcQ^+=\conv\{V^+\},\quad \mcQ^-=\conv\{V^-\},\]
where $V^+=[v_{1}\, \ldots\, v_{i-1}\, v_{i+1}\, \ldots\, v_{n}\, w]$ and $V^-=[v_{1}\, \ldots\, v_{j-1}\, v_{j+1}\, \ldots\, v_{n}\, w]$.
\textcolor{black}{Figure~\ref{fig:branching_simplex} illustrates an example of this construction when $n=3$.}


\textcolor{black}{
The 
overall branch-and-bound scheme is presented in Algorithm~\ref{alg:bnb_simplex}. 
By the same argument as in Section~\ref{sec:adaptive_grid}, one can show that for any prescribed tolerance $\epsilon_{\text{tol}}>0$, 
Algorithm~\ref{alg:bnb_simplex} terminates. If the algorithm returns a nonnegative lower bound $L$, then $Q$ is certified to be copositive. If it returns a point $\bar{x}\in\Delta^{n-1}$ with $\bar{x}^TQ\bar{x}<0$, then $Q$ is not copositive. In the remaining case, copositivity is inconclusive with the prescribed tolerance.}

\begin{algorithm}[H]
\caption{Spatial Branch-and-Bound Algorithm for Testing Matrix Copositivity}
\label{alg:bnb_simplex}
\begin{algorithmic}
\Require A symmetric matrix $Q\in\reals^{n\times n}$, a tolerance $\epsilon_{\text{tol}}>0$.
\Ensure A lower bound $L$ on the minimum value of $x^TQx$ over $\textcolor{black}{\Delta^{n-1}}$ and a point $\bar{x}\in\textcolor{black}{\Delta^{n-1}}$ such that
$\bar{x}^TQ\bar{x}-L\le \epsilon_{\text{tol}}(1+|L|+|\bar{x}^TQ\bar{x}|)$.
\Initialization
$\mcT\gets\{I_n\}$. 
\Comment{Create a branch-and-bound tree}
\State $L\gets \psi(I_n),\quad U\gets \min_{1\le i\le n} e_i^TQe_i$. 
\Comment{Initialize global lower and upper bounds}
\While{$U-L> \epsilon_{\text{tol}}(1+|L|+|U|)$}
\State $V\gets \argmin_{\hat{V}\in\mcT} \psi(\hat{V}),\quad \mcT\gets \mcT\setminus\{V\}$. 
\State $v_1, \ldots, v_n \gets$ columns of $V$.
\State $(i, j) \gets \operatornamewithlimits{argmax}_{1 \le k, l \le n} \|v_k-v_l\|_2,\quad w \gets (v_i + v_j)/2$.
\State $V^+ \gets [v_1\, \ldots\, v_{i-1}\, v_{i+1}\, \ldots\, v_n\, w],\quad V^- \gets [v_1\, \ldots\, v_{j-1}\, v_{j+1}\, \ldots\, v_n\, w]$. \Comment{Branch}
\State $\mcT \gets \mcT \cup \{V^+,V^-\}$.
\State $L\gets \min_{\hat{V}\in \mcT} \psi(\hat{V})$.  \Comment{Update global lower bound}
\For{$\hat{V}\in\{V^+,V^-\}$}\Comment{Update global upper bound}
\State $x^K \gets$ $K$-th iterate of~\eqref{eq:pgd_copos} with $x^0=w$.
\State $u\gets (x^K)^TQx^K$.
\If{$u<U$} $U\gets u,\quad \bar{x}\gets x^K$.\EndIf
\EndFor
\EndWhile
\State \Return $L, \bar{x}$
\end{algorithmic}
\end{algorithm}


\subsection{Numerical Experiments}\label{sec:num}
In this subsection, we present some numerical experiments focusing on certifying nonnegativity of non-sos polynomials, and solving copositive programs with applications to nonconvex quadratic programming and the maximum clique problem in combinatorial optimization. 
The code and data used for the numerical experiments are available in our GitHub repository.\footnote{\url{https://github.com/YH7422/DisjunctiveSOS}}
\subsubsection{Polynomial Minimization}\label{sec:num_poly}
We apply Algorithm~\ref{alg:bnb} to minimize 
several classical polynomials for which the standard sos test (see, e.g.,~\cite{Parrilo2003}) fails to return the minimum value. The examples, drawn from~\cite{Reznick2000, Delzell1980, ahmadi2017sum}, are listed in the following table. 

\begin{table}
  \centering
  \small
  \begin{tabular}{llll}\toprule
Name & $n$ & $d$ &\hfil Expression\\\midrule
    Motzkin & 3 & 6 & $M_h(x_1,x_2,x_3)=x_1^4x_2^2+x_1^2x_2^4-3x_1^2x_2^2x_3^2+x_3^6$\\
        Robinson-1 & 3 & 6 & \hspace{-1.6ex} $\begin{array}[t]{ll}R_1(x_1,x_2,x_3) =&x_1^6+x_2^6+x_3^6- (x_1^4x_2^2+x_1^2x_2^4+x_1^4x_3^2\\ &+x_1^2x_3^4+x_2^4x_3^2+x_2^2x_3^4)+3x_1^2x_2^2x_3^2\end{array}$\\
  Robinson-2 & 4 & 4& \hspace{-1.6ex} $\begin{array}[t]{ll}R_2(x_1,\ldots,x_4)=&x_1^2(x_1-x_4)^2+x_2^2(x_2-x_4)^2+x_3^2(x_3\\&-x_4)^2+2x_1x_2x_3(x_1+x_2+x_3-2x_4)\end{array}$\\
  Choi-Lam-1 & 4 & 4 & $        CL_1(x_1,\ldots,x_4)=x_1^2x_2^2+x_1^2x_3^2+x_2^2x_3^2+x_4^4-4x_4x_1x_2x_3$\\
  Choi-Lam-2 & 3 & 6 & $CL_2(x_1,x_2,x_3)=x_1^4x_2^2+x_2^4x_3^2+x_3^4x_1^2-3x_1^2x_2^2x_3^2$\\
  Lax & 5&4 &$L(x_1,\ldots,x_5)=\sum_{i=1}^5\prod_{j\neq i}(x_i-x_j)$ \\
  Schm\"udgen & 3 & 6&\hspace{-1.6ex} $\begin{array}[t]{ll}SC(x_1,x_2,x_3)=&200(x_1^3-4x_1x_3^2)^2+200(x_2^3-4x_2x_3^2)^2\\&+(x_2^2-x_1^2)x_1(x_1+2x_3)(x_1^2-2x_1x_3\\&+2x_2^2-8x_3^2)\end{array}$\\
\textsc{Partition} & 6 & 4& $P(x_1,\ldots,x_6)=\left(\sum_{i=1}^5x_i\right)^2x_6^2+\sum_{i=1}^5(x_i^2-x_6^2)^2$\\
  Delzell & 4 & 8 & $D(x_1,\ldots,x_4)=x_1^4x_2^2x_4^2+x_2^4x_3^2x_4^2+x_1^2x_3^4x_4^2-3x_1^2x_2^2x_3^2x_4^2+x_3^{8}$\\
  Stengle-$k$ & 3 &$4k+2$  & $S_k(x_1,x_2,x_3)=x_1^{2k+1}x_3^{2k+1}+(x_2^2x_3^{2k-1}-x_1^{2k+1}-x_1x_3^{2k})^2$\\
  \bottomrule
  \end{tabular}
\end{table}

We record the number of subregions generated by Algorithm~\ref{alg:bnb} with $\epsilon_{\text{tol}}=10^{-4}$, using a single projected gradient descent step in~\eqref{eq:pgd_poly} for computing an upper bound at each node of the branch-and-bound tree. We also experiment with the two types of initialization described in Section~\ref{sec:adaptive_grid}. Table~\ref{tab:num_poly} shows that with both initializations, the algorithm terminates with a small number of subregions in most cases. 


\begin{table}
  \centering
  \caption{The number of subregions produced by Algorithm~\ref{alg:bnb} with two types of initialization described in Section~\ref{sec:adaptive_grid}.}
  \label{tab:num_poly}
  \small
  \begin{tabular}{@{}c@{\hspace{1.5em}}c@{}}
  \begin{minipage}[t]{0.45\textwidth}
    \centering
    \begin{tabular}{lll}
      \toprule
      & \multicolumn{2}{c}{\# subregions} \\
      \multirow{-2}{*}{Name} & Init 1 & Init 2 \\\midrule
      Motzkin    & 4   & 7   \\
      Robinson-1 & 4   & 8   \\
      Robinson-2 & 8   & 19  \\
      Choi-Lam-1 & 5   & 15  \\
      Choi-Lam-2 & 4   & 8   \\
      Lax        & 98 & 149 \\
      Schm\"udgen & 4 & 5 \\\bottomrule
    \end{tabular}
  \end{minipage}
  &
  \begin{minipage}[t]{0.45\textwidth}
    \centering
    \begin{tabular}{lll}
      \toprule
      & \multicolumn{2}{c}{\# subregions} \\
      \multirow{-2}{*}{Name} & Init 1 & Init 2 \\\midrule
          \textsc{Partition} & 69 & 161 \\
        Delzell & 8 & 5 \\
      Stengle-1   & 4 & 10 \\
      Stengle-2   & 4 & 4 \\
      Stengle-3   & 4 & 4 \\
      Stengle-4   & 4 & 4 \\
      Stengle-5   & 4 & 4 \\
       \bottomrule
    \end{tabular}
  \end{minipage}
  \end{tabular}

    \vspace{.5em}
  \begin{minipage}{0.9\textwidth}
    \small \textit{Remark:} The number of subregions can be lowered, sometimes significantly, by a better implementation of the branch-and-bound algorithm. See \url{https://github.com/YH7422/DisjunctiveSOS} for heuristic improvements to our implementation made by Claude Code, which result, e.g., in 64 (Init 1) and 49 (Init 2) regions for the Lax polynomial, and 32 (Init 1) and 35 (Init 2) regions for the \textsc{Partition} polynomial.      
  \end{minipage}
\end{table}

Figures~\ref{fig:poly_R2_PT_init1} and~\ref{fig:poly_R2_PT_init0} show the lower and upper bounds on the minimum values of the Robinson-2 polynomial and the \textsc{Partition} polynomial versus the number of subregions produced by Algorithm~\ref{alg:bnb}.

\begin{figure}[H]
\centering
\includegraphics[width=0.45\textwidth]{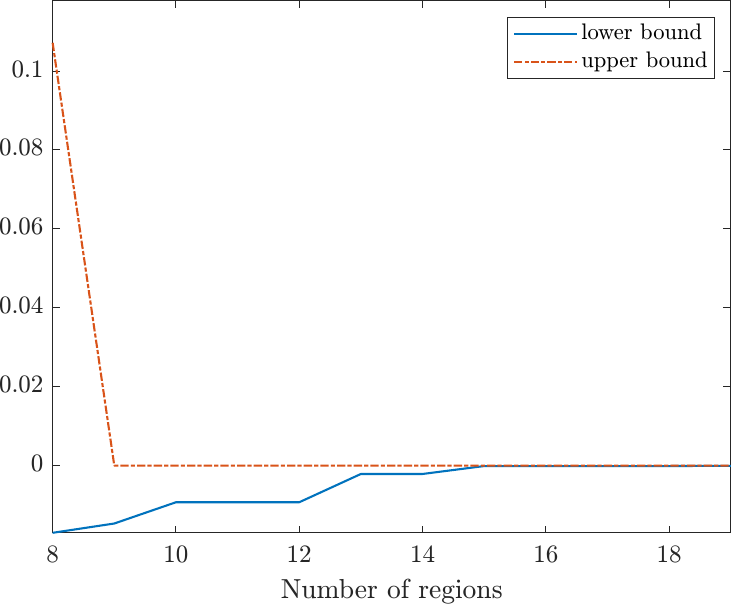}\hfill
\includegraphics[width=0.45\textwidth]{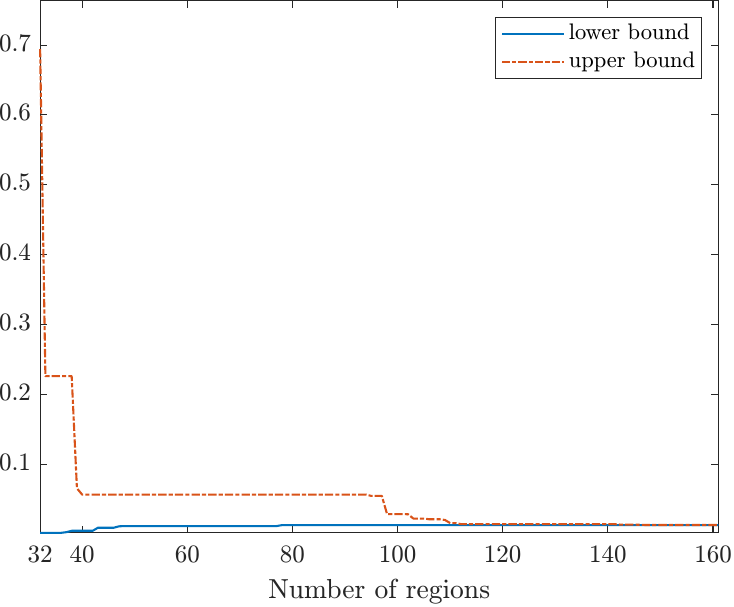}
    \caption{Lower and upper bounds on the minimum value of a given polynomial versus the number of subregions produced by Algorithm~\ref{alg:bnb} with the first type of initialization.
    Left: Robinson-2 polynomial. Right: \textsc{Partition} polynomial.}
\label{fig:poly_R2_PT_init1}
\end{figure}

\begin{figure}[H]
\centering
\includegraphics[width=0.44\textwidth]{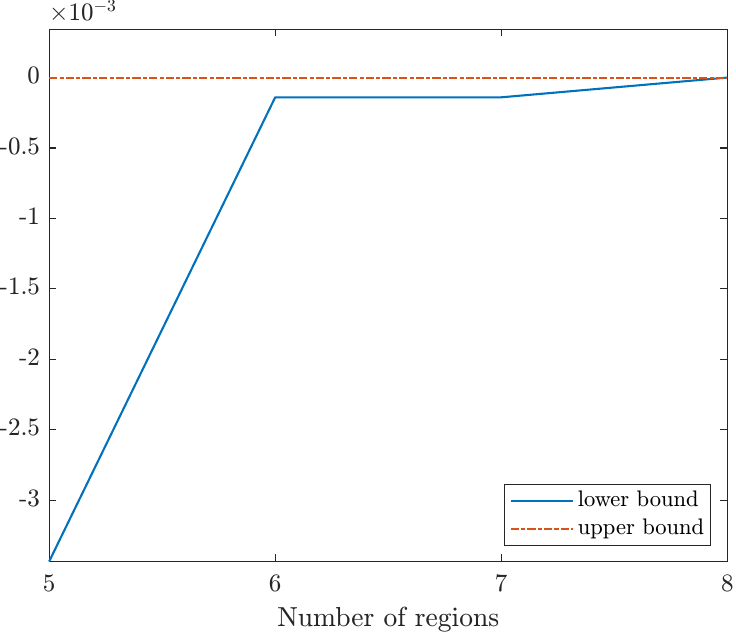}\hfill
\includegraphics[width=0.45\textwidth]{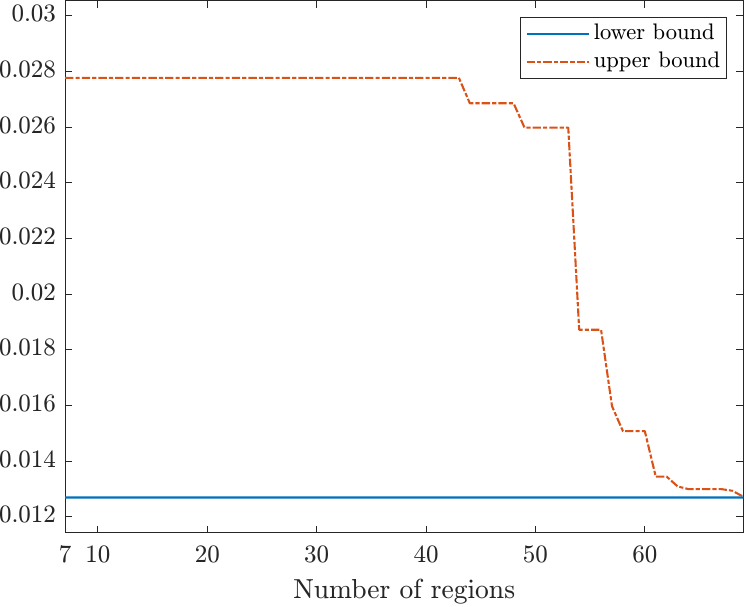}
    \caption{Lower and upper bounds on the minimum value of a given polynomial versus the number of subregions produced by Algorithm~\ref{alg:bnb} with the second type of initialization.
    Left: Robinson-2 polynomial. Right: \textsc{Partition} polynomial.}
\label{fig:poly_R2_PT_init0}
\end{figure}

\subsubsection{Applications to Copositive Programming}\label{sec:num_cop}

We now present some numerical experiments with Algorithm~\ref{alg:bnb_simplex}, which can solve copositive programs or the closely related problem of minimizing a quadratic form over the unit simplex.

\paragraph{Solving nonconvex quadratic \textcolor{black}{programs} (QPs).}

We apply Algorithm~\ref{alg:bnb_simplex} to four nonconvex QP instances of the form $\min_{x\in\textcolor{black}{\Delta^{n-1}}}\; x^TQx$ taken from~\cite{Bundfuss2009}.
The four matrices defining the objective functions of these QPs are as follows:
\allowdisplaybreaks
\begin{alignat*}{4}
Q_1 &= \begin{bmatrix}
1 & 0 & 1 & 1 & 0 \\
0 & 1 & 0 & 1 & 1 \\
1 & 0 & 1 & 0 & 1 \\
1 & 1 & 0 & 1 & 0 \\
0 & 1 & 1 & 0 & 1
\end{bmatrix},
&\quad
Q_2 &= {\footnotesize \left[\begin{array}{llllllllllll}
1 & 0 & 0 & 0 & 0 & 0 & 1 & 1 & 1 & 1 & 1 & 1 \\
0 & 1 & 0 & 0 & 1 & 1 & 0 & 0 & 1 & 1 & 1 & 1 \\
0 & 0 & 1 & 1 & 0 & 1 & 0 & 1 & 0 & 1 & 1 & 1 \\
0 & 0 & 1 & 1 & 1 & 0 & 1 & 0 & 1 & 0 & 1 & 1 \\
0 & 1 & 0 & 1 & 1 & 0 & 1 & 1 & 0 & 1 & 0 & 1 \\
0 & 1 & 1 & 0 & 0 & 1 & 1 & 1 & 1 & 0 & 0 & 1 \\
1 & 0 & 0 & 1 & 1 & 1 & 1 & 0 & 0 & 1 & 1 & 0 \\
1 & 0 & 1 & 0 & 1 & 1 & 0 & 1 & 1 & 0 & 1 & 0 \\
1 & 1 & 0 & 1 & 0 & 1 & 0 & 1 & 1 & 1 & 0 & 0 \\
1 & 1 & 1 & 0 & 1 & 0 & 1 & 0 & 1 & 1 & 0 & 0 \\
1 & 1 & 1 & 1 & 0 & 0 & 1 & 1 & 0 & 0 & 1 & 0 \\
1 & 1 & 1 & 1 & 1 & 1 & 0 & 0 & 0 & 0 & 0 & 1
\end{array}\right]},
\\[1em]
{Q_3} &= {\footnotesize \begin{bmatrix}
-14 & -15 & -16 & 0 & 0 \\
-15 & -14 & -12.5 & -22.5 & -15 \\
-16 & -12.5 & -10 & -26.5 & -16 \\
0 & -22.5 & -26.5 & 0 & 0 \\
0 & -15 & -16 & 0 & -14
\end{bmatrix},}
&\quad
{ Q_4} &= {\footnotesize \begin{bmatrix}
0.9044 & 0.1054 & 0.5140 & 0.3322 & 0 \\
0.1054 & 0.8715 & 0.7385 & 0.5866 & 0.9751 \\
0.5140 & 0.7385 & 0.6936 & 0.5368 & 0.8086 \\
0.3322 & 0.5866 & 0.5368 & 0.5633 & 0.7478 \\
0 & 0.9751 & 0.8086 & 0.7478 & 1.2932
\end{bmatrix}.}
\end{alignat*}

We record the number of subregions generated by Algorithm~\ref{alg:bnb_simplex} with $\epsilon_{\text{tol}}=10^{-6}$. 
We perform 5 projected gradient descent steps in~\eqref{eq:pgd_copos} to obtain upper bounds at each node of the branch-and-bound tree. 
Table~\ref{tab:num_cop} compares our subregion counts against those reported in~\cite{Bundfuss2009} using an LP-based adaptive approximation algorithm. 


\begin{table}
  \centering
    \caption{The number of subregions produced by Algorithm~\ref{alg:bnb_simplex} to solve nonconvex QPs.}
      \label{tab:num_cop}
      \small
  \begin{tabular}{llllll}
    \toprule
 Instance & $Q_1$ & $Q_2$ & $Q_3$ & $Q_4$\\
    \midrule
    Adaptive Linear Approximation~\cite{Bundfuss2009} & 6 & 158& 44 & 27\\
    Algorithm~\ref{alg:bnb_simplex}& 2&42 & 5&17\\\bottomrule
  \end{tabular}
\end{table}

Figure~\ref{fig:num_copositive} shows the lower and upper bounds on the optimal value of the four QPs versus the number of subregions.

\begin{center}
\includegraphics[width=0.45\textwidth]{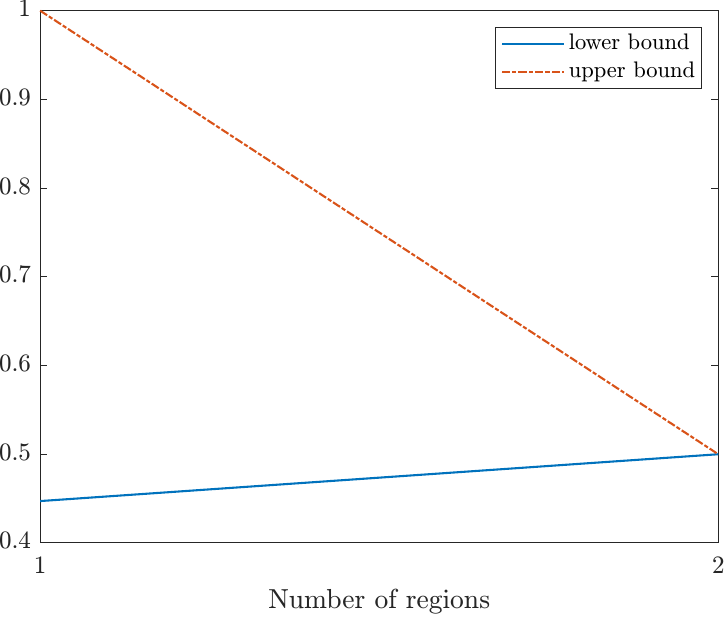}\hfill
\includegraphics[width=0.45\textwidth]{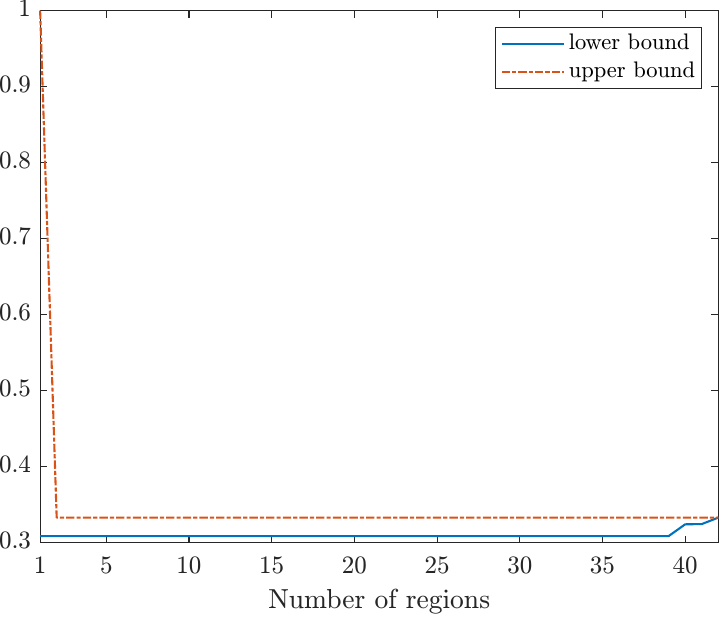}\\
\includegraphics[width=0.45\textwidth]{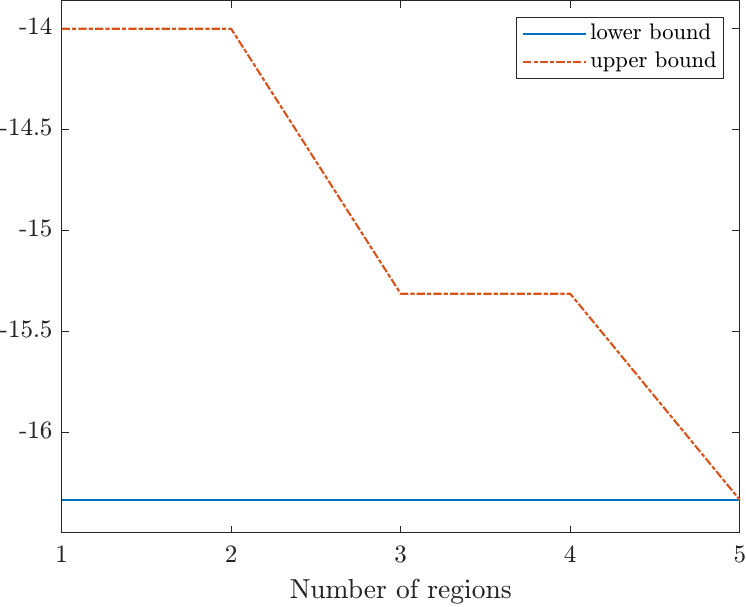}\hfill
\includegraphics[width=0.44\textwidth]{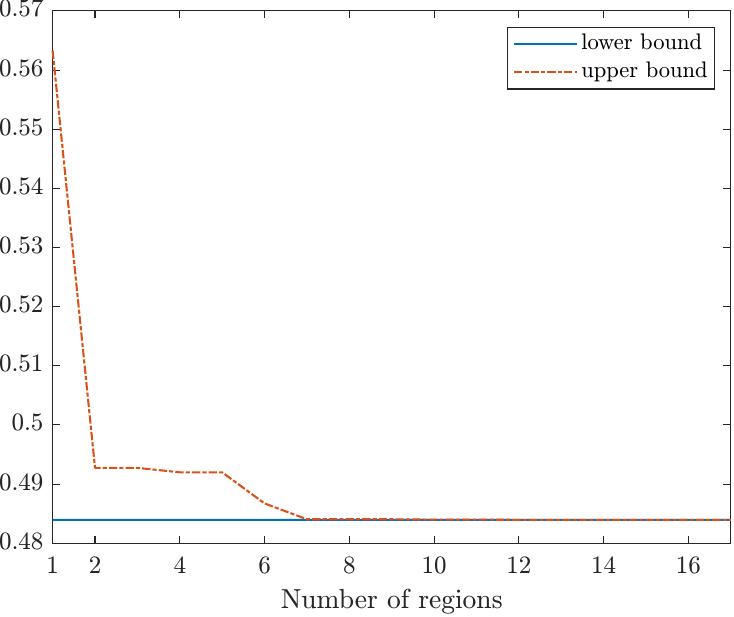}
\captionof{figure}{Lower and upper bounds on the optimal value of the QP $\min_{x\in\textcolor{black}{\Delta^{n-1}}}\; x^TQx$ produced by Algorithm~\ref{alg:bnb_simplex} versus the number of subregions. Top left: $Q=Q_1$. Top right: $Q=Q_2$. Bottom left: $Q=Q_3$. Bottom right: $Q=Q_4$.}
\label{fig:num_copositive}
\end{center}


\paragraph{Computing the clique number of a graph.}
A clique in a graph is a set of pairwise adjacent vertices. The clique number of a graph $G$, denoted by $\omega(G)$, is the size of a maximum
clique in $G$. 
By a theorem of Motzkin and Straus~\cite{Motzkin1965}, we have
\[\frac{1}{\omega(G)}=\min\limits_{x\in\textcolor{black}{\Delta^{n-1}}} x^T(I_n+\bar{A})x,\]
where $\bar{A}$ is the adjacency matrix of the complement graph $\bar{G}$; i.e., the graph obtained by swapping the edges and non-edges of $G$. Equivalently, the clique number can be written as the optimal value of the copositive program
\begin{equation}\label{eq:clique}
\begin{aligned}
    \omega(G)=\ &\min_{k\in\reals}\ k\\
    &\ \st\ k(I_n+\bar{A}) - J_n\in\mcC_n.
\end{aligned}    
\end{equation}
To estimate the clique number by Algorithm~\ref{alg:bnb_simplex}, at each node of the branch-and-bound tree, we perform 10 projected gradient descent steps in~\eqref{eq:pgd_copos} with $Q=I_n+\bar{A}$ to obtain a lower bound on $\omega(G)$, and apply our inner approximation of $\mcC_n$ to~\eqref{eq:clique} to obtain an upper bound on $\omega(G)$. 
Since the clique number is an integer, the algorithm terminates when the ceiled lower bound and the floored upper bound coincide.

We present some experiments on Erd\"os-R\'enyi random graphs~\cite{Erdos1959} $G(n,p)$, where $n$ vertices are connected pairwise and independently with probability $p\in(0,1)$. 
We generate four random instances with $n=75$ and $p=0.5$, and compare our bounds against the Schrijver upper bound $\vartheta'(\bar{G})$ on $\omega(G)$~\cite{Schrijver1979}.\footnote{The Schrijver bound $\vartheta'(\bar{G})$ satisfies $\omega(G)\leq \vartheta'(\bar{G}) \leq \vartheta(\bar{G})$, where $\vartheta$ is the more familiar Lov\'asz theta function~\cite{Lovasz1979}.} 
Table~\ref{tab:clique_random} shows that our algorithm successfully determines the exact clique number in all four instances using a moderate number of subregions.

\begin{table}
\centering
\caption{Lower and upper bounds on the clique number, and the number of subregions produced by Algorithm~\ref{alg:bnb_simplex} for four random instances of $G(75, 0.5)$.}
\label{tab:clique_random}
\small
\begin{tabular}{cccccc}
\toprule
Instance & $\omega(G)$ & $\vartheta'(\bar{G})$ & Upper bound & Lower bound & \# subregions \\
\midrule
1 & 8 & 9.0336 & 8.9014 & 7.9956 & 8 \\
2 & 8 & 9.1274 & 8.9815 & 7.9967 & 48 \\
3 & 8 & 8.9770 & 8.8686 & 7.9950 & 4 \\
4 & 8 & 9.1400 & 8.9952 & 7.9973 & 16 \\
\bottomrule
\end{tabular}
\end{table}

Figure~\ref{fig:num_clique} shows the lower and upper bounds, produced by Algorithm~\ref{alg:bnb_simplex}, on the clique number of the four random instances of $G(75,0.5)$ versus the number of subregions.

\begin{center}
\includegraphics[width=0.45\textwidth]{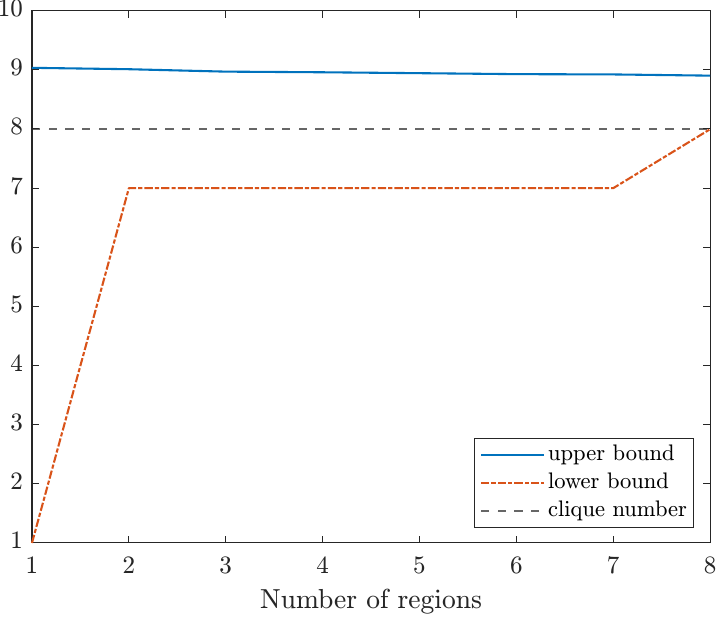}\hfill
\includegraphics[width=0.45\textwidth]{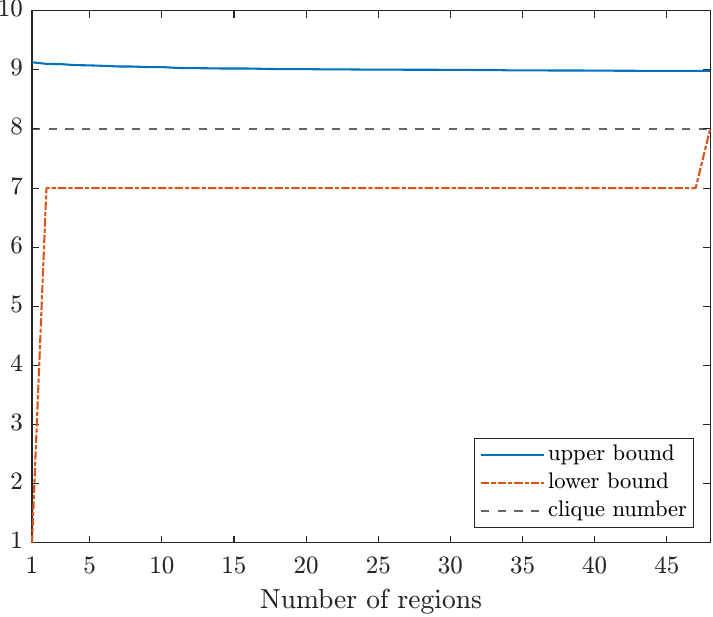}\\[-0.5ex]
\includegraphics[width=0.45\textwidth]{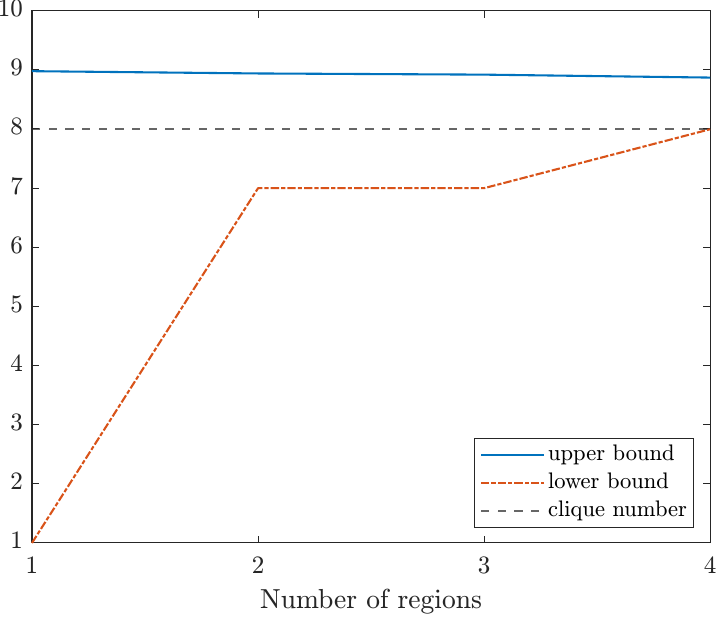}\hfill
\includegraphics[width=0.45\textwidth]{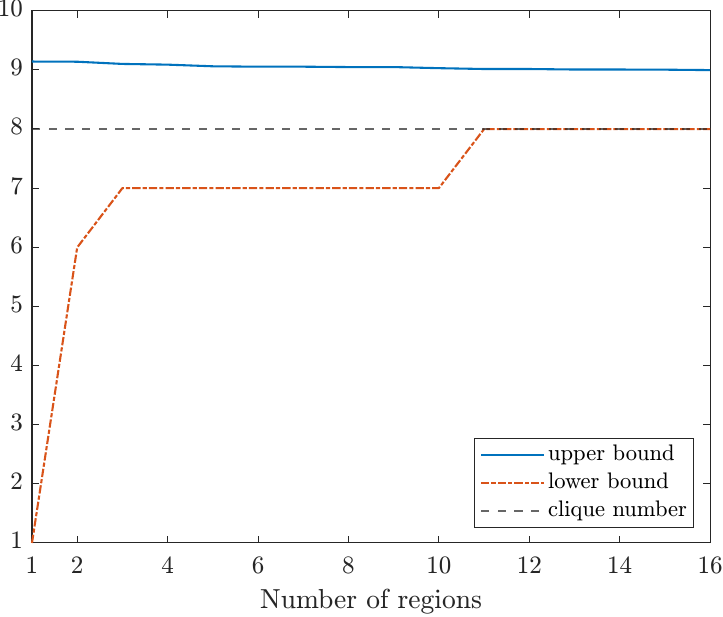}
\captionof{figure}{Lower and upper bounds on the clique numbers of four random instances of $G(75, 0.5)$ produced by Algorithm~\ref{alg:bnb_simplex} versus the number of subregions.}
\label{fig:num_clique}
\end{center}

\section{Conclusions and Future Directions}\label{sec:discussion}
In this paper, we introduced the disjunctive sum of squares method for certifying nonnegativity of polynomials.
Unlike the classical sum of squares approach, which certifies nonnegativity via a single algebraic identity, our method constructs multiple algebraic identities, which together certify nonnegativity. 
We showed that by doing so, the degrees of the sum of squares polynomials involved in the proof (and hence the size of the underlying semidefinite constraints) can be kept low and fixed. We proposed algorithms that can search for disjunctive proofs of nonnegativity, and studied extensions of our method to copositive programs and polynomial optimization problems.

{
This work leaves several interesting directions for future research. While we proposed a general reduction from constrained polynomial optimization problems to unconstrained ones in Section~\ref{sec:hierarchy}, future work could tailor the design of algebraic disjunctions to the geometry of specific feasible sets (as we did, e.g., in the case of the nonnegative orthant in Section~\ref{sec:copositive}). 
Another interesting direction is to develop more efficient implementations of the disjunctive sum of squares approach that in addition to utilizing spatial {branch-and-bound} (cf.\ Section~\ref{sec:bnb}), potentially exploit the natural parallelism across subregions. 
Along these lines, it would be interesting to combine the disjunctive sum of squares method with techniques for exploiting sparsity~\cite{Waki2006, Zheng2022, Zheng2021} and symmetry~\cite{Gatermann2004, Riener2013} in sum of squares programs. Another avenue to study is the dual side of our framework through the theory of moments~\cite{lasserre2001global, Lasserre2019}, which has proven useful for extracting solutions in polynomial optimization problems~\cite{Henrion2005}.
Finally, a promising direction is the design of instance-specific algebraic disjunctions that require fewer subregions than the universal constructions provided in this paper (cf.\ Theorems~\ref{cor:strictly_pos} and~\ref{cor:strictly_pos_new}). The branch-and-bound framework of Section~\ref{sec:bnb} constructs such disjunctions adaptively, but there could be techniques that in addition exploit algebraic properties of the given problem instance.}

{We end with an open problem of theoretical interest: does there exist a nonnegative form that admits no disjunctive sum of squares proof of nonnegativity with any number of identities? This question has two natural variants: one where the degree of the sum of squares proof is required to be the same as {that of} the given form, and one where it is allowed to be larger. By Theorem~\ref{cor:strictly_pos}, any such form must have a zero.
A natural candidate would be a polynomial with a {``bad point''}; see, e.g.,~\cite{Reznick2005, Benoist2021} and references therein. 
An example of such a polynomial is Delzell's polynomial $D$ (see~\cite{Delzell1980})
$$D(x_1,\ldots,x_4)=x_1^4x_2^2x_4^2+x_2^4x_3^2x_4^2+x_1^2x_3^4x_4^2-3x_1^2x_2^2x_3^2x_4^2+x_3^{8},$$
which was also a test case for us in Section~\ref{sec:num_poly}. It is known that while $D$ is nonnegative, $\left(\sum_{i=1}^4 x_i^2\right)^k D(x_1,\ldots,x_4)$ is not sos for any integer $k \ge 0$, which implies that the hierarchy of SDPs based on Reznick's Positivstellensatz~\cite{reznick1995uniform} cannot certify nonnegativity of~$D$. 
It turns out, however, that the disjunctive sum of squares approach can produce a low-degree proof of nonnegativity of~$D$ only with two subregions: 
\ifpreprint
\[
\begin{aligned}
D(x_1,\ldots,x_4) &= (x_1^2x_2x_4 - x_1x_3^2x_4)^2 + (x_2^2x_3x_4 - x_1x_2x_3x_4)^2 + (x_3^4)^2 + 2x_1x_2(x_1x_3x_4-x_2x_3x_4)^2\\
    & = (x_1^2x_2x_4 + x_1x_3^2x_4)^2 + (x_2^2x_3x_4 + x_1x_2x_3x_4)^2 + (x_3^4)^2 - 2x_1x_2(x_1x_3x_4+x_2x_3x_4)^2.
    \end{aligned}
\]
\fi
}

\section*{Acknowledgments}
We are grateful to Abraar Chaudhry, Georgina Hall, and Jeffrey Zhang for insightful discussions.

\ifpreprint
\printbibliography
\fi

\end{document}